\documentclass[11pt]{amsart}
\usepackage[utf8]{inputenc}
\usepackage{csquotes} % goes with babel
\usepackage[english]{babel}
\usepackage[margin=1in]{geometry}

\usepackage{amsmath}
\usepackage{amssymb}
\usepackage{amsfonts}
\usepackage{amsthm}
\usepackage{mathtools}

\usepackage[pdfencoding=auto,hidelinks,psdextra]{hyperref}
\usepackage[capitalize,nameinlink]{cleveref}
\usepackage{xcolor}
\hypersetup{colorlinks=false}

\usepackage{appendix}
\theoremstyle{plain}
    \newtheorem{theorem}{Theorem}
    \newtheorem{lemma}[theorem]{Lemma}
    \newtheorem{corollary}[theorem]{Corollary}
    
    \newtheorem{proposition}[theorem]{Proposition}
    \Crefname{lemma}{Lemma}{Lemmas}
    \Crefname{corollary}{Corollary}{corollaries}
    \Crefname{conjecture}{Conjecture}{conjectures}
    \Crefname{proposition}{Proposition}{Propositions}
\theoremstyle{definition}
    \newtheorem{definition}[theorem]{Definition}
    \newtheorem{obs}[theorem]{Observation}
    
    \newtheorem*{notation}{Notation}
    \newtheorem{example}[theorem]{Example}
    \Crefname{definition}{Definition}{Definitions}
    \Crefname{example}{Example}{Examples}
\theoremstyle{remark}
    \newtheorem{remark}[theorem]{Remark}
    \newtheorem{claim}[theorem]{Claim}
    \Crefname{remark}{Remark}{Remarks}
    \Crefname{claim}{Claim}{Claims}

\numberwithin{equation}{section}
\numberwithin{theorem}{section}

\usepackage{mathrsfs}
\usepackage{indent first}
\usepackage{enumitem}
\usepackage{upgreek}
\usepackage{physics}
\usepackage{stmaryrd}

\usepackage{float}
\usepackage{graphicx}
\usepackage[labelfont=bf,labelsep=period,justification=raggedright]{caption}
\usepackage{subcaption}

\newcommand{\F}{\texorpdfstring{$F$}{F}}
\newcommand{\thetaF}{\texorpdfstring{$\theta(F)$}{\texttheta(F)}}

\DeclareMathOperator*{\argmax}{arg\,max}
\DeclareMathOperator{\ex}{ex}
\DeclareMathOperator{\degd}{deg_d}
\DeclareMathOperator{\degu}{deg_u}
\newcommand{\floor}[1]{\left\lfloor#1\right\rfloor}
\newcommand{\ceil}[1]{\left\lceil#1\right\rceil}
\newcommand{\EE}{\mathbb E}
\renewcommand{\P}{\mathbf P}
\newcommand{\ZZ}{\mathbb Z}
\newcommand{\RR}{\mathbb R}
\newcommand{\eps}{\varepsilon}
\newcommand{\mcB}{\mathcal B}
\newcommand{\mcE}{\mathcal E}
\newcommand{\mcF}{\mathcal F}
\newcommand{\mcH}{\mathcal H}
\newcommand{\mcM}{\mathcal M}
\newcommand{\undirect}[1]{\widetilde{#1}}
\newcommand{\undirectedchi}[1]{\upchi(#1)}
\newcommand{\degw}{\deg_{\rho}}
\newcommand{\ecundir}[1]{e_{\mathrm{u}}\left(#1\right)}
\newcommand{\ecdir}[1]{e_{\mathrm{d}}\left(#1\right)}
\newcommand{\ecw}[2]{w_{#1}\left(#2\right)} % \ecw{\rho}{G}
\newcommand{\simplex}[1]{\triangle^{#1 - 1}}
\newcommand{\simplexx}[1]{\triangle^{#1}}
\newcommand{\sym}[2]{{#1}_{#2}^\mathrm{sym}} % \sym{A}{\rho}
\newcommand{\MG}[2]{#1\left\llbracket#2\right\rrbracket} % \MG{A}{\vb x}
\newcommand{\MMG}[3]{\mathfrak{G}_{#2,#1}^{(#3)}} % \MMG{A}{\rho}{n}
\newcommand{\optvec}[2]{\vb{y}^*_{#2, #1}} % \optvec{A}{\rho}
\newcommand{\optvecn}[3]{\vb{x}^{(#3)}_{#2, #1}} % \optvecn{A}{\rho}{n}
\newcommand{\augmentup}{\xrightarrow{\operatorname{aug}}}
\newcommand{\densesub}{\xrightarrow{\operatorname{sub}}}

\DeclareFontFamily{U}{mathx}{}
\DeclareFontShape{U}{mathx}{m}{n}{<-> mathx10}{}
\DeclareSymbolFont{mathx}{U}{mathx}{m}{n}
\DeclareMathAccent{\widecheck}{0}{mathx}{"71}

\begin{document}
\title{Tur\'an Problems for Mixed Graphs}
\author{Nitya Mani and Edward Yu}

\maketitle

\begin{abstract}
    We investigate natural Tur\'an problems for mixed graphs, generalizations of graphs where edges can be either directed or undirected. We study a natural \textit{Tur\'an density coefficient} that measures how large a fraction of directed edges an $F$-free mixed graph can have; we establish an analogue of the Erd\H{o}s-Stone-Simonovits theorem and give a variational characterization of the Tur\'an density coefficient of any mixed graph (along with an associated extremal $F$-free family). 
    
    This characterization enables us to highlight an important divergence between classical extremal numbers and the Tur\'an density coefficient. We show that Tur\'an density coefficients can be irrational, but are always algebraic; for every positive integer $k$, we construct a family of mixed graphs whose Tur\'an density coefficient has algebraic degree $k$.
\end{abstract}

%%%%%%%%%%%%%%%%%%%%%%%%
%%    INTRODUCTION    %%
%%%%%%%%%%%%%%%%%%%%%%%%

\section{Introduction}\label{sec.intro}
One of the earliest results in extremal graph theory (or indeed, classical combinatorics) was a result of Mantel in 1907, which showed that an $n$-vertex graph which does not contain a triangle has at most $\frac{n^2}4$ edges~\cite{mantel}.
A central problem in extremal graph theory has been to determine the \textit{extremal number} $\ex(n, F)$ of $F$-free graphs, the maximum number of edges that an $n$-vertex graph can have if it does not contain a copy of $F$ as a subgraph. This question is referred to as Tur\'an's problem, due to his famous result in 1941 characterizing $\ex(n, K_r)$ where $K_r$ is a \textit{clique} or \textit{complete graph} on $r$ vertices~\cite{turan_thm}. Tur\'an's theorem was generalized to all non-bipartite forbidden graphs $F$ in the following result: 

\begin{theorem}[Erd\H os-Stone-Simonovits, 1946~\cite{erdos_simonovits_1966, erdos_stone_1946}]
Let $\upchi(F)$ denote  the \emph{chromatic number} of $F$, the minimum number of colors required in a proper vertex-coloring of $F$ (where no two adjacent vertices have the same color). Then, we have that
\[\ex(n,F) = \left(1 - \frac{1}{\upchi(F)-1} + o(1) \right) \frac{n^2}{2}.\]
\end{theorem}

The above theorem asymptotically resolves Tur\'an's problem whenever $F$ is not bipartite (i.e., $\upchi(F) > 2$). The case of bipartite $F$ is still widely open and a very active area of research. K\"ov\'ari-S\'os-Tur\'an \cite{Kovari1954} gave a general upper bound that is believed to be tight for complete bipartite graphs $K_{s, t}$; the bound is known to be tight for $K_{2,2}, K_{3,3}$ \cite{bipartite-survey} and was recently shown to be tight when $t$ is sufficiently large relative to $s$ in breakthrough work leveraging randomized algebraic constructions \cite{bukh_bipartite}; see \cite{bipartite-survey} for a comprehensive survey of bipartite Tur\'an results.

Motivated by this line of inquiry, a more general, fundamental family of questions in extremal combinatorics are the so-called \emph{Tur\'an-type questions}, which seek to understand how large a graph-like object can be when we forbid certain substructures from appearing. Perhaps the most well-known such question is the hypergraph Tur\'an problem, a rich and fertile area of research which has attracted much attention and where even very simple-seeming questions are wide open; see \cite{keevash} for a survey of such extremal hypergraph results. Other Tur\'an-type problems of substantial interest include the rainbow Tur\'an problem (on edge-colored graphs) \cite{rainbow-turan} and the directed-graph Tur\'an problem \cite{brown_erdos_simonovits}.

In this work, we study natural Tur\'an problems for \emph{mixed graphs}, graph-like objects that can have both directed and undirected edges.

\begin{definition}\label{d:mixedgraph}
    A \emph{mixed graph} $G = (V, E)$ consists of vertex set $V$ and edge set $E$. Each edge $e \in E$ is a pair of distinct vertices $u, v \in V$,
    that is either \emph{undirected} (denoted $e = uv = vu$) or \emph{directed} (denoted $e = u\widecheck v = \widecheck v u$).
    For a directed edge $u\widecheck v$, the vertex $v$ is the \emph{head} vertex of the edge, and $u$ the \emph{tail}. There is at most one edge on any pair of vertices.
    Two vertices are considered adjacent if they are connected by any edge, regardless of whether the edge is directed or undirected.

    Let $\ecundir{G}$ and $\ecdir{G}$ be the number of undirected and directed edges in mixed graph $G$ respectively, and let $\alpha(G) = \frac{\ecundir{G}}{\binom{v(G)}2}$ and $\beta(G) = \frac{\ecdir{G}}{\binom{v(G)}2}$ be the \textit{undirected} and \textit{directed} \textit{edge densities} of $G$, respectively.
\end{definition}

Mixed graphs naturally appear in a variety of enumerative, combinatorial, and theoretical computer science applications.
Mixed graphs have been previously studied in a variety of contexts including spectral graph theory \cite{mixedgraph_adjmatrix1, mixedgraph_unitgain} and the extremal degree-diameter problem \cite{mixedgraph_degreediameter,mixedgraph_geodeticturan}.
They also arise naturally in the context of theoretical computer science, including the extremal problems of coloring \cite{mixedgraph_coloring1, mixedgraph_coloring2, mixedgraph_edgecoloring} and job scheduling \cite{mixedgraph_scheduling}; they have useful applications in object classification and labeling, social network models, and inference on Bayesian networks \cite{mixedgraph_multilabel,mixedgraph_objectclass,mixedgraph_pdnetworks,mixedgraph_networks}.
Specific Tur\'an problems on special families of mixed graphs (such as geodetic mixed graphs in~\cite{mixedgraph_geodeticturan}) have also been explored in recent work.
Here, we make a systematic study of a natural Tur\'an type quantity on mixed graphs, leveraging the below notion of subgraphs that arises in recent applications e.g.~\cite{mani_ksat,mani_ksat2}.

\begin{definition}
We say a mixed graph $F$ is a \emph{subgraph} of $G$ ($F\subseteq G$) if $F$ can be obtained from $G$ by deleting vertices, deleting edges, and forgetting edge directions. We say $G$ is \emph{$F$-free} if $F$ is not a subgraph of $G$ (see~\Cref{fig.subgraph}).
\end{definition}

\begin{figure}[ht]
    \centering
    \includegraphics[width=5cm]{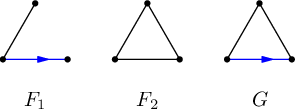}
    \caption{Both $F_1$ and $F_2$ are subgraphs of $G$; $F_1$ is not a subgraph of $F_2$.}
    \label{fig.subgraph}
\end{figure}

In an initial attempt to study extremal questions on mixed graphs, one might first consider a ``naive'' extremal problem,  maximizing the edge density $\alpha(G)+\beta(G)$ over all $F$-free $G$. 
However, this notion is uninteresting for mixed graphs; for all $F$ with at least one directed edge, $F \not\subseteq K_n$ and $\alpha(K_n) + \beta(K_n) = 1$, hence the maximal value of $\alpha(G)+\beta(G)$ over $F$-free $G$ would always be $1$. 

The above motivates the question we study here, trying to understand the \textit{tradeoff} between directed and undirected edges in $F$-free mixed graphs. More formally, we study the following extremal quantity:

\begin{definition}
    Let $F$ be a mixed graph. The \emph{Tur\'an density coefficient} $\theta(F)$ is the maximum $\rho$ such that
    \[
    \alpha(G) + \rho\beta(G) \le 1 + o(1)
    \]
    over all $F$-free $n$-vertex mixed graphs $G$.
    If $\beta(G) = o(1)$ over $F$-free $n$-vertex $G$, we say $\theta(F)=\infty$.
\end{definition}

Note that for all mixed graphs $F$, it is true that $\theta(F)\geq1$, because $\alpha(G) + \beta(G) \le 1$ for all mixed graphs $G$ (every pair of vertices is connected by at most one edge).

We verify that this notion is well-defined in~\Cref{prop.theta-exists}. Throughout this article, we will primarily focus on the ``dense'' case where $\theta(F) < \infty$, as opposed to the ``degenerate'' case where $\theta(F) = \infty$ and $F$ is bipartite.

The above extremal quantity naturally arises in theoretical computer science problems, notably the Bollob\'as-Brightwell-Leader conjecture about the number of distinct $k$-SAT functions (posed in 2003~\cite{BBL03}) was recently resolved for all $k \ge 2$ in a two-part argument~\cite{mani_ksat} that first reduced this open problem to a bound on the Tur\'an density coefficient of a triangle like mixed hypergraph (using a suitable generalization of the above mixed graph notion) and then subsequently gave a sufficiently strong bound on the associated Tur\'an density coefficient to prove the conjecture~\cite{mani_ksat2} (the $k = 2$ case of the Bollob\'as-Brightwell-Leader conjecture, originally proved in~\cite{AL07}, can be reduced to proving that if $F$ is a triangle with a single directed edge, then $\theta(F) > \log_2 3$). Other seemingly unrelated questions can also be restated as Tur\'an problems on mixed graphs; for example, extremal bounds on the number of edges in the square of an undirected graph (as in~\cite{FUR92}) can be reduced to showing that $\theta(F) = 2$ for the mixed triangle described above.

In our work, we completely characterize the Tur\'an density coefficients of mixed graphs, observing some similarities to Tur\'an densities of undirected graphs but also several key differences, which in part motivate the study of mixed Tur\'an coefficients. We will see for instance, that there is a natural oriented analogue of being bipartite (\cref{def.uncollapsible}) and that we can use this notion in conjunction with the chromatic number of the underlying undirected graph to give an analogue of the Erd\H{o}s-Stone-Simonovits theorem for mixed graphs.
%, exhibiting bounds on $\theta(F)$ in terms of $\upchi(F),$ which we define as the chromatic number of the underlying undirected graph of $F$.

\begin{definition}\label{def.uncollapsible}
A mixed graph is \emph{uncollapsible} if some two head vertices are adjacent or some two tail vertices are adjacent. Otherwise, it is \emph{collapsible}.
\end{definition}
Roughly speaking, collapsible mixed graphs will contain a large blowup of at least one directed edge (see \Cref{sec.ess_analogue} for details).

\begin{theorem}\label{mixedgraph-general}
Let $F$ be a mixed graph. Then, if $\upchi(F)$ is the chromatic number of the underlying undirected graph of $F$, we have the following:
    \begin{enumerate}[label=(\roman*)]
    \item  $\theta(F) = 1$ if and only if $F$ is uncollapsible.
    \item $\theta(F)=\infty$ if and only if $F$ admits a proper 2-coloring of the vertices such that all head vertices are the same color. Otherwise $\theta(F)\leq 2$.
    \item If $F$ has at most one directed edge, then 
    \[\theta(F)=1+\frac1{\undirectedchi{F}-2}.\]
    \item Otherwise, we have that
        \[1+\frac1{\undirectedchi{F}} \leq \theta(F) \leq 1+\frac1{\undirectedchi{F}-2}.\]
    \end{enumerate}
\end{theorem}

\begin{remark}
  At first glance, it may be unclear if the above result is tight in general. In fact, both ends of the above inequality can be attained by infinitely many mixed graphs $F$; the general upper bound is achieved by mixed graphs with a single directed edge (as noted in case (iii)) and a tight instance of the lower bound is given in~\Cref{mixedgraph-general-lowerbound}.
\end{remark}

The above Erd\H{o}s-Stone-Simonovits-style result falls short of a characterization of the Tur\'an density coefficient $\theta(F)$. However, we are able to completely determine $\theta(F)$ with a variational characterization of $\theta(F)$ that determines the Tur\'an density coefficient as the minimum of a finite number of solutions to finite dimensional optimization problems. We give an informal statement below and defer the formal statement to~\cref{thm.theta-variational}.

\begin{theorem}[Informal]
Let $F$ be a mixed graph with at least one directed edge where $\theta(F) \in(1,\infty)$. There is some family $\mcB_F$ of mixed graphs on at most $v(F)$ vertices, such that 
  \[\theta(F) = \min_{B\in\mcB_F} \left\{\min_{\vb y \in \simplex{v(B)}}\left\{\frac{1-\vb y^\intercal U_B \vb y}{\vb y^\intercal D_B \vb y}\right\} \right\},\]
where $B$ is represented by $(U_B, D_B)$, an ordered pair of matrices that encode the undirected and directed edges of $B$, and $\simplex{v(B)}$ is the simplex in $\RR^{v(B)}$. 
\end{theorem}

To show the above result, we leverage and generalize a technique of Brown, Erd\H{o}s, and Simonovits~\cite{brown_erdos_simonovits}.
Our result sets the Tur\'an coefficient of mixed graphs in sharp contrast to more general Tur\'an type problems; for instance, graph homomorphism inequalities are in general not even necessarily \textit{decidable}~\cite{hatami-norime}, which is a very different situation than the finite-time algorithm to verify bounds on $\theta(F)$ implied by the above representation.

The above expression nevertheless seems somewhat complicated, and one might hope that a simple modification of~\cref{mixedgraph-general} might be sufficient to understand $\theta(F)$ in general.
However, the extremal behavior of general mixed graphs turns out to be much more nuanced than the case of undirected graphs.
In \Cref{thm.irrational-theta} we construct a relatively simple mixed graph $F$ with irrational density coefficient,
in contrast with the Erd\H os-Stone-Simonovits theorem which implies graph Tur\'an densities are of the form $1 - 1/k$ for positive integer $k$.
Based on this, one might wonder how complicated $\theta(F)$ can be.
The complexity of Tur\'an-type quantities in discrete structures varies considerably. 
Tur\'an densities of graphs are always rational, as noted above, as are digraphs \cite{brown_erdos_simonovits_2}.
On the other hand, there is a range of results on the existence of hypergraphs with irrational Tur\'an densities and related concepts. For instance, \cite{lagrangian_density} constructed a relatively simple example of a $3$-graph with irrational Lagrangian densities; it was recently shown in \cite{hypergraph-irrational} that there exist hypergraphs with irrational Tur\'an densities; and Tur\'an densities of families of hypergraphs can have arbitrarily high algebraic degree \cite{pikhurko-alg-degree, Pikhurko_possible_densities};
and families of multigraphs can even have transcedental Tur\'an densities \cite{multigraphs-transcedental}.
Here, we show that Tur\'an density coefficients of mixed graphs are algebraic:

\begin{theorem}\label{thm.mixedgraph-algebraic}
    Let $F$ be a mixed graph such that $\theta(F) < \infty$. Then $\theta(F)$ is an algebraic number.
\end{theorem}

Furthermore, every algebraic degree is actually attained in the Tur\'an density coefficient of some finite family of mixed graphs (where for family $\mcF$ of mixed graphs, $G$ is $\mcF$-free if $G$ is $F$-free for all $F\in \mcF$). This suggests a particularly rich structure encoded by the Tur\'an density coefficients of mixed graphs.
\begin{theorem}\label{thm.alg-degree}
For every positive integer $k$ there exists a finite family of mixed graphs having Tur\'an density coefficient with algebraic degree $k$.
\end{theorem}

It remains an interesting question: does every algebraic degree arise as the Tur\'an density coefficient of a \textit{single} mixed graph? In other words, for every positive integer $k$ does there exist a mixed graph $F$ such that $\theta(F)$ has algebraic degree $k$?
\newline

\subsection*{Organization}
We begin by introducing some notation and basic results concerning the Tur\'an problem on mixed graphs in~\Cref{sec.prelims}. In~\Cref{sec.ess_analogue}, we focus our attention on proving~\Cref{mixedgraph-general}, a weaker analogue of the Erd\H{o}s-Stone-Simonovits theorem of extremal graph theory. In~\Cref{sec.variational}, we develop a variational characterization of $\theta(F)$, leveraging this representation several times in~\Cref{sec.algebraic-degree} to study the algebraic degree of $\theta(F)$.

%%%%%%%%%%%%%%%%%%%%%%%%
%%   PRELIMINARIES    %%
%%%%%%%%%%%%%%%%%%%%%%%%
\section{Preliminaries}\label{sec.prelims}

\begin{notation}\label{notation}
    We let $[r] = \{1, \ldots, r\}$. We use $\ZZ_{\geq0}$ and $\ZZ_{>0}$ to denote the set of nonnegative and (strictly) positive integers, respectively; also, $\RR_{\geq0}$ and $\RR_{>0}$ denote the nonnegative and (strictly) positive reals, respectively. We use $\vb 1$ to denote the vector of all $1$s; dimension will be clear from context.
    
    For mixed graph $F$, we let $\undirect{F}$ denote the \emph{underlying undirected graph of $F$}, obtained by forgetting the directions of all directed edges (but retaining every edge).
    For $r \in \ZZ_{>0}$, we use $V = V_1 \sqcup \cdots \sqcup V_r$ to refer to the disjoint union of sets $V_1, \ldots, V_r.$
    
    Asymptotic notation will always refer to the limit for large $n$. We will say that $f(n) = o(g(n))$ if $\lim_{n \to \infty} f(n)/g(n) = 0$ (and in particular $f(n) = o(1)$ if $\lim_{n\to\infty}f(n) = 0$); we say $f(n) = O(g(n))$ if there exists constant $C > 0$ such that $f(n) \le C g(n)$ for sufficiently large $n$; and $f(n) = \Omega(g(n))$ if there exists constant $C > 0$ such that $f(n) \ge C g(n)$  for sufficiently large $n$.
\end{notation}

\subsection{Mixed graph fundamentals}\label{subsec.fundamentals}
Recall from~\Cref{d:mixedgraph} that a mixed graph has edges that can be undirected or directed (c.f.~\Cref{fig.mixedgraphs}).
    \begin{figure}[ht]
    \centering
    \begin{subfigure}{.4\textwidth}
    \centering
    \includegraphics[width=3cm]{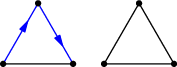}
    \caption*{Mixed graphs}
    \end{subfigure}
    ~
    \begin{subfigure}{0.4\textwidth}
    \centering
    \includegraphics[width=3cm]{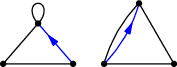}
    \caption*{Not mixed graphs}
    \end{subfigure}
    \caption{The first two graphs are examples of mixed graphs (note that a fully undirected graph is still a mixed graph); the last two are not (self-loops and vertices connected by multiple edges are disallowed).}
    \label{fig.mixedgraphs}
    \end{figure}

Note if $F\subseteq G$, then $\theta(F) \geq \theta(G)$. We first verify that $\theta(F)$ is well defined.

\begin{obs}[Existence of $\theta(F)$]\label{prop.theta-exists}
    Fix mixed graph $F$. If there exists $\eps > 0$ and a sequence $\{G_n\}$ of distinct $F$-free mixed graphs with $\beta(G_n) \ge \eps$, then there exists a maximum $\rho$ such that 
    \[
    \limsup_{n\to\infty} \alpha(G_n) + \rho \cdot \beta(G_n) \leq 1.
    \]
\end{obs}
\begin{proof}
    Let $f(\rho) \coloneqq \limsup_{n\to\infty} \alpha(G_n) + \rho \beta(G_n)$.
    Note that $f(\rho)$ is nondecreasing and bounded between $0$ and $\max(1, \rho)$. Since $\beta(G) \le 1$ for any mixed graph $G$, $f$ is continuous in $\rho$. 
\end{proof}

Many of the standard tools for analyzing extremal graphs extend to mixed graphs. We first observe a supersaturation analogue of~\cite{supersaturation}.

\begin{definition}
    For mixed graph $G$ and $S \subseteq V(G)$, the \emph{induced subgraph} $G[S]$ is the mixed graph with vertex set $S$ and edge set comprising edges in $G$ with both endpoints in $S$.
\end{definition}
\begin{lemma}[Supersaturation]\label{prop.kpdg-supersat}
Fix $\eps > 0$. Given mixed graph $F$ with $\theta(F)<\infty$, there exists constant $c = c(\eps, F)>0$ such that any $n$-vertex mixed graph $G$ with
    $\alpha(G) + \theta(F) \beta(G)\geq 1+\eps$
    must contain at least $c\cdot n^{v(F)}$ copies of $F$ for $n$ sufficiently large.
\end{lemma}
\begin{proof}
Choose $n_0 \in \ZZ_{>0}$ such that any mixed graph $H$ on $n_0$ vertices with
    $\alpha(H) + \theta(F) \beta(H)\geq 1+\eps/2$ contains $F$ as a subgraph.
    Fix integer $n>n_0$ and any mixed graph $G$ on $n$ vertices where
    $\alpha(G) + \theta(F) \beta(G)\geq 1+\eps$.
    Let $S$ be a subset of $n_0$ vertices from $V(G)$ chosen uniformly at random, so 
    \begin{equation*}\label{eq.1-plus-epsilon}
    \EE_S \big[\alpha(G[S]) + \theta(F)\beta(G[S]) \big] = \alpha(G) + \theta(F) \beta(G)\geq 1+\eps.
    \end{equation*}
    Note that for random variable $X$ supported in $[0, \ell]$ and $\delta\in [0,\ell]$, we have $\P[X \ge \EE[X] - \delta ] \ge \delta/\ell$. Thus,
    \begin{equation}\label{eq.s-sub-h}
    \P\left[\alpha(G[S]) + \theta(F)\beta(G[S])\geq 1+\frac{\eps}2\right] \geq\frac{\eps}{2\theta(F)}.
    \end{equation}
    When $\alpha(G[S]) + \theta(F)\beta(G[S])\geq 1+\frac{\eps}2$, by assumption $G[S]$ must contain a copy of $F$ as a subgraph and thus if $T\subseteq S$ is a uniformly random subset of size $v(F)$, then 
    $\P\big[F \subseteq G[T]\big] \geq 1/{\binom{n_0}{v(F)}}.$ Since every copy of $F$ appears in at most $\binom{n - v(F)}{n_0 - v(F)}$ subsets $S$, the number of copies of $F$ in $G$ is at least 
    \[\frac{\eps}{2\theta(F) \binom{n_0}{v(F)}} \cdot \frac{\binom{n}{n_0}}{\binom{n - v(F)}{n_0 - v(F)}}=
    \Omega_{F, \eps}\left(n^{v(F)}\right).\]
\end{proof}

\begin{definition}\label{def.blowup-mixed}
Given a mixed graph $F$ and an integer $t\geq2$, let $F[t]$ denote the \emph{balanced $t$-blowup} of $F$, obtained by replacing each vertex $v_i\in V(F)$ by $t$ copies $v_{i,1}, \dots, v_{i, t}$; each undirected edge $v_i v_j$ with $t^2$ undirected edges $v_{i,k}v_{j,\ell}$ ($1\leq k, \ell \leq t$); and each directed edge $v_i \widecheck v_j$ with $t^2$ directed edges
$v_{i,k}\widecheck v_{j,l}$ ($1\leq k, l \leq t$) (see~\Cref{fig.ex2}).
\end{definition}
\begin{figure}[ht]
\centering
\includegraphics[width=6cm]{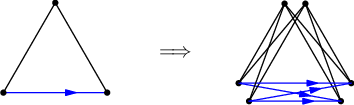}
\caption{A balanced $2$-blowup of a mixed graph}
\label{fig.ex2}
\end{figure}

\begin{lemma}[Blowups]\label{prop.mixedgraph-blowup}
Let $F$ be a mixed graph with $\theta(F)<\infty$ and $t \in \ZZ_{>0}$. Then $\theta(F) = \theta(F[t])$.
\end{lemma}
\begin{proof}
    Suppose otherwise; fix integer $t\geq2$ and mixed graph $F$ on $r$ vertices $w_1, \ldots, w_r$ where $\theta(F)<\infty$ and $\theta(F)\neq\theta(F[t])$.
    Since $F\subseteq F[t]$, it is true that $\theta(F)> \theta(F[t])$.
    Consequently, for some $\eps > 0$, there are $F[t]$-free mixed graphs $G$ on $n$ vertices for arbitrarily large $n$ where 
    \[\alpha(G) + \theta(F) \beta(G) > 1+\eps.
    \] 
    \Cref{prop.kpdg-supersat} implies that $G$ contains at least $c n^r$ copies of $F$ for some $c>0$. By randomly equitably partitioning $V(G) = V_1 \sqcup \cdots \sqcup V_r,$ we can find an equitable $r$-partition of $V(G)$ where the number of copies of $F$ in $G$ with $w_i \in V_i$ is at least $c' n^r$ where $c'=c/(2r)^r$ (each copy of $F$ in $G$ has each $w_i \in V_i$ with probability at least $1/(2r)^r$).

    Consider auxiliary $r$-partite, $r$-uniform hypergraph $\mcH$ with $V(\mcH) = V_1 \sqcup \cdots \sqcup V_r$ where hyperedge $(v_1, \ldots, v_r) \in E(\mcH)$ if and only if the $r$ vertices form a copy of $F$ in $G$ with $v_i \in V_i$. By construction $\mcH$ has at least $c' n^r$ edges. The result then follows by considering the corresponding hypergraph blowup result (c.f. Theorem 2.2 of~\cite{keevash}), noting that if $K_r^{r}$ is the $r$-uniform hypergraph with $r$ vertices and a single edge, a $K_r^r[t]$-free $r$-uniform hypergraph has $o(n^r)$ edges for all positive integers $t$. Hence, for $n$ sufficiently large, $\mcH$ contains a copy of $K_r^r[t]$ yielding a corresponding copy of $F[t]$ in $G$, a contradiction.
\end{proof}

\subsection{Bipartite and undirected mixed graphs}\label{subsec.simple}

As a warmup, we first discuss two very simple families of mixed graphs. A \emph{bipartite} mixed graph is a mixed graph with $\upchi(F) \leq 2$.

We establish our first results on $\theta(F)$, for some simple families $F$. These will also be useful later on in~\Cref{sec.ess_analogue}. For $a,b \in \ZZ_{>0}$, let $K_{a, b}$ denote the undirected complete bipartite graph with parts of size $a$ and $b$. Let $K_{\overrightarrow{a,b}} = (A \sqcup B, E)$ denote the mixed graph with underlying undirected graph $K_{a, b}$ (with $|A| = a, |B| = b$) such that all edges in  $K_{\overrightarrow{a,b}}$ are directed $A \to B$. By applying the Erd\H os-Stone-Simonovits theorem, we obtain an analogue of the classical observation that undirected graphs free of fixed bipartite subgraph $H$ must have $o(n^2)$ edges; we find that mixed graphs free of some fixed subgraph $F \subseteq K_{\overrightarrow{a,b}}$ can only have $o(n^2)$ directed edges. 

\begin{proposition}\label{prop.ktt}
    For mixed graph $F$, we have $\theta(F) = \infty$ if and only if $F \subseteq K_{\overrightarrow{t,t}}$ for some $t$.
    Otherwise $\theta(F)\leq 2$.
\end{proposition}
\begin{proof}
    The main effort is to show that $\theta(F) = \infty$ for $F = K_{\overrightarrow{t,t}}$ for $t \in \ZZ_{>0}$. Fix any $\delta > 0$ and suppose to the contrary there was a sequence $\{G_n\}$ of $F$-free mixed graphs with $v(G_n) \rightarrow \infty$ and $\beta(G_n) > \delta$ for all $n \in \ZZ_{>0}$. 
    Fix $n_0$ such that any undirected graph with at least $2 n_0$ vertices and $\frac{1}{2} n_0^2$ edges contains a copy of $K_{t, t}$ (such $n_0$ exists by the Erd\H{o}s-Stone-Simonovits theorem).
    Note that the subgraph $G_\text{d} \subset G$ comprising only directed edges of $G$ has at least $\delta \binom{n}{2}$ edges. Consequently, by the Erd\H os-Stone-Simonovits theorem, for $n$ sufficiently large, $\widetilde{G}_{\text{d}}$ contains a copy of $K_{n_0, n_0}$ with vertex set $A \sqcup B$. Let $G' \subseteq G_{\text{d}}[A \sqcup B]$ be the bipartite mixed graph comprising the directed edges in $G$ that have tail vertex in $A$ and head vertex in $B$. Without loss of generality, $e(G') \ge \frac{1}{2} n_0^2$. By our choice of $n_0$, $\widetilde{G'}$ contains a copy of $K_{t, t}$. By our choice of edge orientations, this yields a copy of $K_{\overrightarrow{t,t}}$ in $G' \subseteq G$ as desired.
    
    For the other direction, suppose $F\not\subseteq K_{\overrightarrow{t,t}}$ for all $t$. Then, since $\lim_{t\to\infty}\alpha(K_{\overrightarrow{t,t}}) = 0$ and $\lim_{t\to\infty}\beta(K_{\overrightarrow{t,t}}) = \frac12$, we have $\theta(F) \leq 2 < \infty$.
\end{proof}

We can also consider the case where $F$ is a mixed graph with no directed edges; here, we obtain the equivalent to the Erd\H{o}s-Stone-Simonovits theorem by a simple manipulation of definitions.
\begin{obs}\label{prop.undirected}
    Let $F$ be an undirected graph. Then 
    \[\theta(F)=
    \begin{cases}
    \frac{\upchi(F) - 1}{\upchi(F)-2}, & \text{ if } \upchi(F)>2,     \\
    \infty,                            & \text{ if } \upchi(F)\leq 2.
    \end{cases}\]
\end{obs}
\begin{proof}
The case $\upchi(F) \leq 2$ follows from~\Cref{prop.ktt}. Thus, assume $\upchi(F)\geq3$. Mixed graph $G$ is $F$-free if and only if $\undirect{G}$ is $F$-free. Thus, for $F$-free $G$, by the Erd\H os-Stone-Simonovits Theorem
\[\ecundir{G} + \ecdir{G} = e\big(\undirect{G}\big)\leq \ex(n,F) \leq \left(\frac{\upchi(F)-2}{\upchi(F)-1} + o(1)\right)\frac{n^2}2.\]
Rewriting in terms of $\alpha(G)$ and $\beta(G)$ then gives the desired lower bound.

Similar to the undirected case, the extremal $F$-free mixed graphs for undirected $F$ are the family of mixed graphs with all edges directed and underlying undirected graph $T(n, \upchi(F) - 1)$, the $(\upchi(F)-1)$-partite Tur\'an graph on $n$ vertices.
\end{proof}

%%%%%%%%%%%%%%%%%%%%%%%%
%%   PROOFS OF THMS   %%
%%%%%%%%%%%%%%%%%%%%%%%%

\section{Proof of \texorpdfstring{\Cref{mixedgraph-general}}{Theorem 1.6}}\label{sec.ess_analogue}
\subsection{Uncollapsible mixed graphs} Recall from~\Cref{def.uncollapsible} that a mixed graph is \textit{uncollapsible} if it has two adjacent head vertices, or two adjacent tail vertices, else \textit{collapsible}; examples of uncollapsible configurations are given in~\Cref{fig.uncollapsible}. 
For collapsible mixed graph $F$, we partition $V(F) = V_0\sqcup V_h\sqcup V_t$, where $V_h$ are vertices that are the head of some edge(s), $V_t$ are vertices that are tails of some edge(s), and $V_0$ are the remaining vertices. Note that $V_h, V_t$ are both independent sets.

Let $\overrightarrow{K_r}[t]$ be the balanced $t$-blowup of the mixed graph $\overrightarrow{K_r}$ obtained by directing a single edge of an otherwise undirected $K_r$. An equivalent rephrasing of the above discussion is that $F$ is collapsible if and only if $F\subseteq \overrightarrow{K_r}[t]$ for some $r,t \in \ZZ_{>0}$. This yields the following operation that motivates our terminology.

\begin{figure}[ht]
\centering
\begin{minipage}{0.5\textwidth}
  \centering
  \includegraphics[width=5cm]{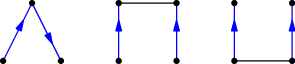}
  \caption{Three examples of uncollapsible graphs}
  \label{fig.uncollapsible}
  \end{minipage}\begin{minipage}{0.5\textwidth}
  \centering
  \includegraphics[width=6cm]{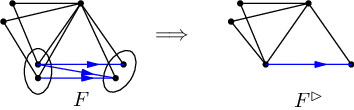}
  \caption{Head-tail collapsing mixed graph $F$}
  \label{fig.ex-collapsion}
  \end{minipage}
\end{figure}

\begin{definition}\label{def.collapsion}
If $F$ is collapsible and has at least one directed edge,
we perform \emph{head-tail collapsing} on $F$ to obtain mixed graph $F^{\rhd}$ by considering partition $V(F) = V_0 \sqcup V_h \sqcup V_t$ and contracting $V_h$ into a single vertex $h$ and $V_t$ into a single vertex $t$ (see~\Cref{fig.ex-collapsion} for an illustration).
If $F$ is a mixed graph  with no directed edges then define $F^\rhd$ to be $F$.
\end{definition}

We first observe that $\theta(F) = 1$ for uncollapsible mixed graphs $F$ by constructing a family of $F$-free mixed graphs with many directed edges.

\begin{definition}\label{def.graph-mxn}
For $n \in \ZZ_{>0}$ and $x \in (0, 1)$, let $M(x,n)$ denote the following mixed graph with $n$ vertices, constructed in three steps:
\begin{enumerate}[label=(\roman*)]
    \item Construct an undirected complete graph $X$ on $\floor{nx}$ vertices;
    \item Construct an independent set $Y$ with $\ceil{n(1-x)}$ vertices;
    \item Draw a directed edge from every vertex in $X$ towards every vertex in $Y$.
\end{enumerate}
This construction is illustrated in \Cref{fig.bxn}.
\end{definition}

\begin{figure}[ht]
\centering
\includegraphics[width=6cm]{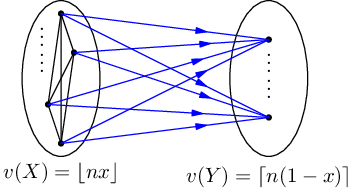}
\caption{$M(x, n)$}
\label{fig.bxn}
\end{figure}

\begin{proposition}\label{prop.mxn-theta}
    Let $F$ be a mixed graph. If $F\not\subseteq M(x,n)$ for all $n \in \ZZ_{>0}, x\in(0,1)$, then $\theta(F) = 1$.
\end{proposition}
\begin{proof}
Note that $\theta(F) \ge 1$ for all mixed graphs $F$. Suppose to the contrary that there was $\eps > 0$ such that $\theta(F) \ge 1 + \eps$.
Consider $M(x, n)$ for fixed $x$ as $n$ grows large. Observe that 
     \begin{align*}
        \lim_{n\to\infty} \alpha\big(M(x,n)\big) & = \lim_{n\to\infty} \tbinom{\floor{nx}}2/\tbinom n2 = x^2;                 \\
        \lim_{n\to\infty} \beta\big(M(x,n)\big)  & = \lim_{n\to\infty} {\floor{nx}\cdot\ceil{n(1-x)}}/{\tbinom n2} = 2x(1-x).
    \end{align*}
    Since $M(x,n)$ is $F$-free $\alpha(M(x, n)) + \theta(F) \beta(M(x, n)) \le 1 + o(1)$ for all $n \in \ZZ_{>0}, x \in (0, 1),$ and thus
    \[\theta(F) \leq \lim_{n\to\infty}\frac{1-\alpha(M(x,n))}{\beta(M(x,n))} +o(1) = \frac{1-x^2}{2x(1-x)} = \frac{1+x}{2x}.\]
    When $x > 1/(1 + \eps)$, the above implies that $\theta(F) \le 1 + \eps/2$, a contradiction.
\end{proof}

\begin{corollary}\label{prop.uncollapsible}
Let $F$ be an uncollapsible mixed graph. Then $\theta(F)=1$.
\end{corollary}
\begin{proof}
Note that $M(x, n)$ does not have any adjacent head vertices, and thus if two head vertices in $F$ are adjacent, then $F\not\subseteq M(x,n)$ for all $x, n$, so $\theta(F)=1$ by~\Cref{prop.mxn-theta}. If tail vertices in $F$ are adjacent, the result follows analogously to above, by repeating the argument of~\Cref{prop.mxn-theta} with $M'(x, n)$, the mixed graph obtained by reversing the directions of all directed edges in $M(x,n)$.
\end{proof}

\subsection{Mixed graphs with one directed edge}\label{subsec.one-dir}
Our goal in this section is to determine the Tur\'an density coefficient for mixed graphs $F$ with exactly one directed edge. It turns out this quantity will be $\theta(F) = 1+1/(\upchi(F)-2)$. We will reduce this case to $F = \overrightarrow{K_r}$, an $r$-clique with exactly one directed edge, then compute $\theta(\overrightarrow{K_r})$ by a further reduction to the undirected setting. Let $t(n, r)$ denote the number of edges in the \textit{Tur\'an graph} $T(n, r)$, the densest $n$-vertex $K_{r+1}$-free undirected graph.

For vertex $v\in V(G)$ define its \emph{undirected degree} $\degu v$ as the number of vertices connected to $v$ by undirected edges, and its \emph{directed degree} $\degd v$ as the number of vertices connected to $v$ by directed edges.

\begin{proposition}\label{prop.complete-one-edge-turan}
    Let $n\geq r\geq 2$ with $n, r \in \ZZ_{>0}$. For all $n$-vertex $\overrightarrow{K_{r+1}}$-free mixed graphs $G$, we have $\alpha(G)+\frac{\binom n2}{t(n,r)}\beta(G)\leq 1$.
\end{proposition}
\begin{proof}
Let mixed graph $G$ maximize $\ecundir{G} + \frac{\tbinom n2}{t(n,r)}\ecdir{G}$ over all $n$-vertex $\overrightarrow{K_{r+1}}$-free graphs. We apply Zykov symmetrization~\cite{zykov_symm} to show that if vertices $a,b$ are not adjacent in $F$, we may assume $a$ and $b$ have identical neighborhoods (i.e., for all vertices $v$ either $va, vb$ are both nonedges or are identically oriented with respect to $v$).
    \begin{claim}[Zykov Symmetrization]\label{claim.zykovsymm}
        Suppose vertices $a,b$ are not adjacent in $G$. Then there exists mixed graph $G'$ with $V(G') = V(G)$ such that $G'$ is $\overrightarrow{K_{r+1}}$-free, $a,b$ have the same neighborhood in $G'$, and $$\ecundir{G'} + \frac{\tbinom n2}{t(n,r)}\ecdir{G'} \ge \ecundir{G} + \frac{\binom n2}{t(n,r)}\ecdir{G}.$$
    \end{claim}
    \begin{proof}
    Without loss of generality, assume $\degu(a) + \frac{\binom n2}{t(n,r)} \degd(a)
    \leq \degu(b) + \frac{\binom n2}{t(n,r)} \degd(b).$
    Delete $a$ and replace it with a perfect copy $b'$ of $b$.
    This weakly increases $\ecundir{G} + \frac{\binom n2}{t(n,r)}\ecdir{G}$,
    and it cannot introduce a copy of $\overrightarrow{K_{r+1}}$, since $G$ was $\overrightarrow{K_{r+1}}$-free, and $b, b'$ are not adjacent.
    \end{proof}
    By \Cref{claim.zykovsymm} we can assume that every pair of non-adjacent vertices $a,b$ have equal degree, else by symmetrization we may increase the edge count of the mixed graph by deleting the one with smaller degree and replacing it with a copy of the other.

    We may now iterate the process in \Cref{claim.zykovsymm} as follows: label the vertices $v_1, \dots, v_n$ in some fixed order. We then iterate over all pairs $(v_i, v_j)$ for $1 \leq i < j \leq n$ over all $i$ in an outer loop ($1 \leq i < n$) and over all $j$ in an inner loop ($i+1 \leq j \leq n$), at each step replacing $v_j$ with a copy of $v_i$ if they are not connected. This does not change the edge count of the mixed graph because all nonadjacent vertices have identical degree; furthermore, at any iteration $(i,j)$ it is true that any nonadjacent vertices $v_{i'}, v_{j'}$ with $(i',j')$ lexicographically less than or equal to $(i,j)$ are copies of each other, which means by the end of the process every pair of non-adjacent $a,b$ has identical neighborhood.

    Let $A$ be a maximum independent set of $G$ ($|A| \ge 1$); every $v\in V(G)\setminus A$ is connected to at least one vertex in $A$; thus by \Cref{claim.zykovsymm}, each $v\in V(G)\setminus A$ is connected to all vertices in $A$, with $va$ identically oriented for all $a\in A$.
    Partition $V(G) = A\sqcup B\sqcup C$, where $B$ is the set of vertices connected to $A$ by directed edges
    and $C$ is the set of vertices connected to $A$ by undirected edges (see \Cref{fig.ex.zykov}). Every pair $u \in B$ and $v \in C$ must be adjacent since $u$ and $v$ do not have identical neighborhoods. 
    \begin{figure}[ht]
        \centering
        \includegraphics[width=3.5cm]{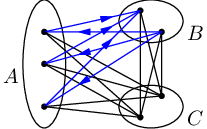}
        \caption{Partitioning the vertices of $G$}
        \label{fig.ex.zykov}
    \end{figure}
    
        We are now ready to prove the desired claim; it suffices to show that for any $G$ as above, $\ecundir{G} + \frac{\tbinom n2}{t(n,r)}\ecdir{G} \le \binom{n}{2}.$
        We show this inequality by induction on $r$. We first consider the base case $r=2$; note that $B$ and $C$ cannot both have vertices, else a $\overrightarrow{K_3}$ is formed by choosing one vertex from each of $A,B,C$. Suppose $B$ is empty. There are no directed edges between vertices of $C$ (if $c_1\widecheck c_2$ is such an edge then $\{a,c_1,c_2\}$ for $a\in A$ forms a $\overrightarrow{K_3}$), so $G$ has no directed edges; the inequality then follows since $\ecundir G \leq \tbinom n2$.
    Next suppose $C$ is empty. Similar to above, there cannot be edges between vertices of $B$. Thus, $G$ is bipartite, so by Tur\'an's theorem~\cite{turan_thm} $e(G) \leq t(n,2)$ and the result follows.

    The inductive step is similar: assume that the above inequality is true for $\overrightarrow{K_r}$-free graphs; we will show the inequality continues to hold for $\overrightarrow{K_{r+1}}$-free graphs. We proceed in two cases:
    \begin{itemize}    
    \item[$|B| = 0$:] Note that $C$ is $\overrightarrow{K_r}$-free (if $\overrightarrow{K_r} \subseteq G[C]$ then adding any $a\in A$ forms a $\overrightarrow{K_{r+1}}$), so
    \[
    \ecundir{G[C]} + \frac{\tbinom n2}{t(n,r)} \ecdir{G[C]}
    \leq \ecundir{G[C]} + \frac{\tbinom n2}{t(n,r-1)} \ecdir{G[C]} \leq \binom{|C|}2,
    \]
    where the second inequality follows from the inductive hypothesis. Thus, as desired
    \[\ecundir{G} + \frac{\binom n2}{t(n,r)}\ecdir{G}
    \leq \abs{A}\cdot\abs{C} + \binom{\abs{C}}2 \le \binom n2.\]
    
        \item[$|B| > 0$:] We claim $K_{r+1} \not\subseteq G$. Suppose to the contrary we have $K_{r+1} \subseteq G$ with on vertex set $S \subset V(G).$ If $S$ does not contain a vertex in $A$, then add one arbitrarily; do the same for $B$. Note that the new set $S'$ has size at least $r+1$ and $\widetilde{G[S']}$ is a clique where for each $u \in S' \cap A, v \in S' \cap B$ we have a directed edge between $u$ and $v$ in $E(G[S'])$. This implies $\overrightarrow{K_{r+1}} \subseteq G[S']$, contradiction.
    Thus, by Tur\'an's theorem
    \[\ecundir{G} + \frac{\binom n2}{t(n,r)}\ecdir{G} \leq \frac{\binom n2}{t(n,r)}\left(\ecundir{G}+\ecdir{G}\right) \leq \binom n2.\]
    This completes the induction and the proof.
    \end{itemize}
\end{proof}

\begin{corollary}\label{cor.complete-one-edge-theta}
    For any integer $r\geq3$, we have $\theta(\overrightarrow{K_r}) = 1+\frac1{r-2}$.
\end{corollary}
\begin{proof}
    The lower bound is obtained by taking the limit in \Cref{prop.complete-one-edge-turan} as $n$ grows large, since $t(n,r-1) = (\tfrac{r-2}{r-1}+o(1))\binom n2$.
    The upper bound on $\theta(\overrightarrow{K_r})$ is obtained by construction: let $G$ be any mixed graph with $\undirect{G} = T(n,r-1)$ and all edges directed arbitrarily. Such $G$ is $F$-free with $\alpha(G) = 0, \beta(G) = \frac{r-2}{r-1} - o(1)$.
    This implies $\theta\left(\overrightarrow{K_r}\right)\leq1+\frac1{r-2},$ the desired bound. 
\end{proof}

We can now handle the case of any mixed graph with a single directed edge.

\begin{lemma}\label{lem:onedirectededge}
If mixed graph $F$ has one directed edge with $\upchi(F) \ge 3$, then $\theta(F) = 1 + 1/(\upchi(F) -2).$
\end{lemma}
\begin{proof}
 Let $r = \upchi(F)$ and consider a proper $r$-coloring of $F$ with associated vertex partition $V(F)=V_1\sqcup V_2\sqcup \dots\sqcup V_r$. Without loss of generality, we assume the directed edge of $F$ is from a vertex in $V_1$ to a vertex in $V_2$ (see \Cref{fig.onedir}).
    
  \begin{figure}[ht]
  \centering
  \begin{minipage}[b]{.45\textwidth}
    \centering
    \includegraphics[width=6cm]{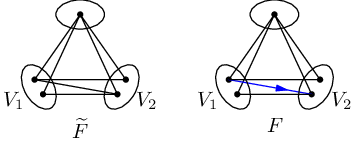}
    \caption{Vertex-coloring of $F$ with one directed edge}
    \label{fig.onedir}
    \end{minipage}\begin{minipage}{.03\textwidth}\;
    \end{minipage}\begin{minipage}[b]{.47\textwidth}
        \centering
    \includegraphics[width=4.5cm]{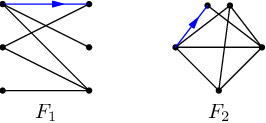}
        \caption{Example mixed graphs with one directed edge}
        \label{fig.onedir-examples}
    \end{minipage}
    \end{figure}
    Applying~\Cref{prop.undirected} we have 
    upper bound $\theta(F) \leq \theta\left(\undirect{F}\right) = 1+\frac1{r-2}$.
    On the other hand, by~\Cref{cor.complete-one-edge-theta}, $\theta\left(\overrightarrow{K_r}\right)= 1+\frac1{r-2}$. 
    Since $F \subseteq \overrightarrow{K_r}[t]$ for some $t$, by \Cref{prop.mixedgraph-blowup},
    \[
    1+\frac1{r-2}
    = \theta\left(\overrightarrow{K_r}\right)
    = \theta\left(\overrightarrow{K_r}[t]\right)
    \leq \theta(F).
    \]
    Combining these gives the desired result.
\end{proof}
\begin{example}\label{ex.onedir-examples}
    In~\Cref{fig.onedir-examples}, the pictured mixed graphs satisfy $\theta(F_1) =\infty$ and $\theta(F_2) = \frac32$ by \Cref{prop.ktt,lem:onedirectededge} respectively.
\end{example} 

\subsection{General mixed graphs}\label{subsec.general}

We now consider collapsible mixed graphs.

\begin{lemma}\label{lem.chi-collapsion-ineq}
Let $F$ be a collapsible mixed graph with at least one directed edge where $\upchi(F^{\rhd}) > 2$. Then,
\[1+\frac1{\upchi(F)} \le 1 + \frac{1}{\undirectedchi{F^{\rhd}} - 2} \leq \theta(F) \leq 1+\frac1{\undirectedchi{F}-2}.\]
\end{lemma}
\begin{proof}
    By definition, $F$ is a subgraph of a balanced $t$-blowup of $F^{\rhd}$ for some sufficiently large $t$. Therefore by \Cref{prop.mixedgraph-blowup}, $\theta(F)\geq \theta(F^{\rhd}[t]) = \theta(F^{\rhd})$.
    Since $F^{\rhd}$ is a mixed graph with exactly one directed edge, by~\Cref{lem:onedirectededge}, we have that 
    \[
    \theta(F)\geq \theta\left(F^{\rhd}\right) = 1+\frac1{\undirectedchi{F^{\rhd}}-2} \ge 1 + \frac{1}{\upchi(F)}.
    \]
The final inequality arises by partitioning $V(F^{\rhd}) = V_0 \sqcup V_h \sqcup V_t$ as in~\Cref{def.collapsion}; $V_0$ can be properly colored with at most $\upchi(F)$ colors and we need (at most) a single new color for each of $V_h$ and $V_t$, so $\upchi(F^{\rhd}) \le \upchi(F) + 2$.
On the other hand, the underlying undirected graph $\undirect{F}$ is a subgraph of $F$. Thus by \Cref{prop.undirected}, we have the opposite inequality
\[\theta(F)\leq \theta\left(\undirect{F}\right)= 1+\frac1{\undirectedchi{F}-2}.\qedhere\]
\end{proof}

The proof of \Cref{mixedgraph-general} now follows immediately from the earlier results.
\begin{proof}[Proof of \Cref{mixedgraph-general}]    
We consider each case in turn:
\begin{enumerate}[label=(\roman*)]
    \item By~\Cref{prop.uncollapsible}, if $F$ is uncollapsible, then $\theta(F) = 1$; the only if direction comes from an exhaustive consideration of the below cases.
    \item Note that $F$ satisfies the condition for this case if and only if $F \subseteq K_{\overrightarrow{t,t}}$ for some $t$; by~\Cref{prop.ktt} this is equivalent to $\theta(F) = \infty$ and the result follows.
    \item If $F$ has at most one directed edge, the desired equality follows by~\Cref{lem:onedirectededge}.
    \item In the remaining case, $F$ is collapsible and the result follows by~\Cref{lem.chi-collapsion-ineq}.
    \qedhere
\end{enumerate}
\end{proof}

It may not be clear at first glance how tight (or far from tight) the bounds of~\Cref{mixedgraph-general} are for collapsible mixed graphs. The upper bound is attained by mixed graphs with at most one directed edge (as noted in~\Cref{prop.undirected,lem:onedirectededge}). The lower bound is also tight in general as the following (easily generalizable) example illustrates.

\begin{example}\label{mixedgraph-general-lowerbound}
\begin{figure}[ht]
  \centering
  \includegraphics[width=12cm]{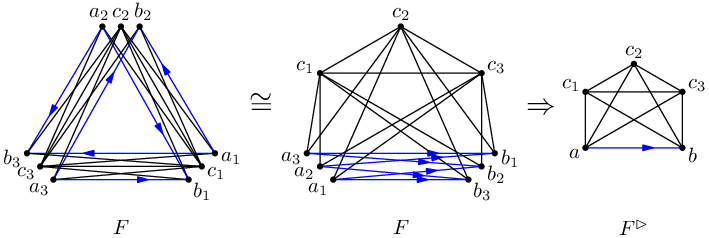}
  \caption{$F$ has $\theta(F) = 1+\frac1{\undirectedchi{F}} = \frac43$.}
  \label{lowerbound-ex}
\end{figure}
For mixed graph $F$ as pictured in~\Cref{lowerbound-ex} with vertices labeled as $V(F) = \{a_1, b_1, c_1, a_2, b_2, c_2, a_3, b_3, c_3\}$, we show that $\theta(F) = 1+\frac1{\undirectedchi{F}} = \frac43$, achieving the lower bound of~\Cref{mixedgraph-general}. To show the matching upper bound on $\theta(F)$, we construct a family of mixed graphs with many directed edges that is $F$-free.

\begin{figure}[ht]
\centering
\begin{minipage}{.3\textwidth}
  \centering
  \includegraphics[width=4cm]{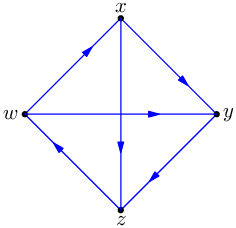}
  \captionsetup{width=0.6\linewidth}
  \caption{Mixed graph $G$}
  \label{fig.k4_for_333}
\end{minipage}%
\begin{minipage}{.35\textwidth}
  \centering
  \includegraphics[width=4.5cm]{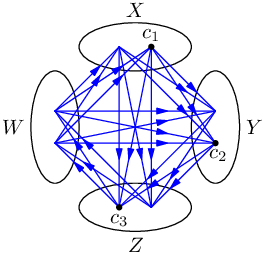}
  \captionsetup{width=0.6\linewidth}
  \caption{Case (i)}
  \label{fig.k4t_for_333_case1}
\end{minipage}%
\begin{minipage}{.35\textwidth}
  \centering
  \includegraphics[width=4.5cm]{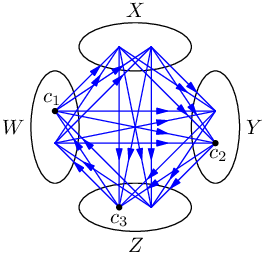}
  \captionsetup{width=0.7\linewidth}
  \caption{Case (ii)}
  \label{fig.k4t_for_333_case2}
\end{minipage}
\end{figure}

Let $G$ be the mixed graph in~\Cref{fig.k4_for_333}. we will show that $F\not\subseteq G[t]$ for all $t \in \ZZ_{>0}$. Suppose to the contrary $F\subseteq G[t]$ for some $t \in \ZZ_{>0}$.
We partition $V(G[t]) = X \sqcup Y \sqcup Z \sqcup W$.
Because $c_1, c_2, c_3$ form a triangle in $F$, in any copy of $F$ in $G$, the images of $c_1, c_2, c_3$ must be in different parts of $G[t]$. We then arrive at a contradiction by checking the following two cases (without loss of generality, by symmetry in $G$), and finding that we cannot find images for all of $V(F)$ in $G[t]$ that preserve $E(F)$.
\begin{enumerate}[label=(\roman*)]
\item  $c_1, c_2, c_3$ are in $X, Y, Z$, respectively, forming a transitive tournament on $3$ vertices,
see \Cref{fig.k4t_for_333_case1};
\item  $c_1, c_2, c_3$ are in $W, Y, Z$, respectively, forming a directed cycle,
see \Cref{fig.k4t_for_333_case2}.
\end{enumerate}

Thus, $F\not\subseteq G[t]$ for all $t \in \ZZ_{>0}$. 
Since $\alpha(G[t])=0$ and $\beta(G[t]) = 6 t^2 / \binom{4t}2 \rightarrow \frac34$ as $t \rightarrow \infty$, we find the desired matching upper bound $\theta(F) \leq \frac43$. 
\end{example}

%%%%%%%%%%%%%%%%%%%%%%%%%%%%%%%%%
%%  VARIATIONAL density coeff  %%
%%%%%%%%%%%%%%%%%%%%%%%%%%%%%%%%%

\section{A variational characterization of \thetaF}\label{sec.variational}
This section gives a variational characterization of $\theta(F)$, showing that $\theta(F)$ is the solution to a finite dimensional optimization problem. Essentially, we show that $\theta(F)$ is the minimum of a function over a finite set of \textit{mixed adjacency matrices}, defined precisely below.
This characterization gives an explicit method of computing $\theta(F)$ for a general given mixed graph $F$, in contrast to~\Cref{mixedgraph-general} which may not provide a tight bound for a specific mixed graph $F$. In~\Cref{sec.algebraic-degree}, we leverage this characterization to study the algebraicity and algebraic degree of $\theta(F).$ 

\subsection{A weighted extremal problem}\label{subsec.weightedextremal}
We begin by considering a slightly different extremal problem on mixed graphs; we will try to maximize a weighted edge count over $F$-free mixed graphs but fix the weight of directed edges in advance. 

\begin{definition}
For $\rho \in (1,\infty)$ and mixed graph $G$, we define the \emph{$\rho$-weighted edge count} of $G$ as $\ecw\rho G :=\ecundir G + \rho \cdot \ecdir G$, i.e., we count weighted edges where undirected edges have weight $1$ and directed edges have weight $\rho$. The \emph{$\rho$-density} is given by $w_{\rho}(G)/\tbinom{v(G)}{2}$.
\end{definition}

In this section, we are motivated by the following question: given a forbidden mixed graph $F$ and fixed $\rho \in (1, \infty)$, what is the asymptotic maximum of the $\rho$-density of a sequence of distinct $F$-free graphs?
This question is closely tied to the problems studied in the previous section; 
in fact, $\theta(F)$ is the largest value of $\rho$ such that for every sequence $\{G_n\}$ of distinct $F$-free graphs,
\begin{equation}\label{eq.theta_equiv}
\limsup_{n\to\infty}\;\frac{\ecw{\rho}{G_n}}{\binom {v(G_n)}2} \le 1.    
\end{equation}
We later employ this fact in our variational characterization of $\theta(F)$. Throughout this section, we consider $\rho \in (1, \infty)$ to be a fixed constant.

The primary result of this section,~\Cref{thm.matrix-maximal}, establishes that we can construct a family of asymptotically maximal (in $\rho$-density) family of mixed graphs by carefully selecting a sufficiently dense but small mixed graph and taking an (asymmetric) blowup. The subsequent analysis adapts the primary technique used in previous work of Brown, Erd\H os, and Simonovits to study extremal problems in directed graphs~\cite{brown_erdos_simonovits}. Our primary objects of study will be \emph{mixed adjacency matrices}, which resemble hypergraph Lagrangians~\cite{keevash} and patterns~\cite{Pikhurko_possible_densities} in form and purpose.

\begin{definition}[Mixed adjacency matrix]\label{def.mixed-adj}
A \emph{mixed adjacency matrix} $A = (U, D)$ is an ordered pair of $r \times r$ matrices where $U$ and $D$ satisfy the following conditions:
  \begin{enumerate}[label=(\roman*)]
    \item $U$ is a symmetric matrix in $\{0, 1\}^{r \times r}$.
    \item $D_{ij}\in\{0,2\}$ for all $i,j\in[r]$, and $D_{ij}\neq 0$ implies that $D_{ji} = 0$. Thus, $D_{ii} = 0$ for all $i\in[r]$.
    \item For all $i,j \in [r]$, at most one of $D_{ij}$ and $U_{ij}$ is nonzero.
  \end{enumerate}
  We say that $r$ is the \emph{size} of $A$,
  that $U$ is the \emph{undirected part} of $A$,
  and that $D$ is the \emph{directed part} of $A$.
\end{definition}

A mixed adjacency matrix $(U,D)$ of size $r$ can be thought of as a ``template'' for constructing associated mixed graphs; the elements of $U$ and $D$ specify the type and direction of edges between the $r$ parts. We make this precise below.
\begin{definition}[Mixed matrix graphs]\label{def.mat-graph}
  Let $A=(U,D)$ be a size $r$ mixed adjacency matrix and fix
  $\vb{x}=( x_1,\dots,x_r ) \in \ZZ_{\geq0}^r$.
Define the \emph{mixed matrix graph} $A\llbracket\vb{x}\rrbracket$ as the mixed graph with vertex set given by the disjoint union $W_1 \sqcup\dots\sqcup W_r$ (such that
  $\abs{W_i} = x_i$ for each $i\in[r]$), and edge set given by the following collection of vertex pairs:
  \begin{enumerate}[label=(\roman*)]
    \item For each $i\in [r]$, 
          \begin{itemize}
            \item If $U_{ii}=1$, the induced subgraph of $A\llbracket\vb{x}\rrbracket$ on $W_i$ is an undirected clique.
            \item If $U_{ii} = 0$, there are no edges with both endpoints in $W_i.$
          \end{itemize}
    \item For each $i,j\in[r]$ with $i\neq j$:
        \begin{itemize}
            \item If $U_{ij}=U_{ji}=1$, for each $w_i \in W_i, w_j \in W_j$ we have undirected edge $w_i w_j.$
            \item If $D_{ij}=2$ and $D_{ji}=0$, for each $w_i \in W_i, w_j \in W_j$ we have directed edge $w_i \widecheck w_j.$
            \item Otherwise ($U_{ij}=U_{ji}=D_{ij}=D_{ji}=0$), we have no edges between $W_i$ and $W_j.$
            \end{itemize}
  \end{enumerate}
\end{definition}

\begin{example}\label{ex.mat-graph-1}
  \Cref{fig.mat-graph-ex} illustrates the mixed matrix graph $\MG A{\vb{x}}$ for $A=(U,D)$,
  \[U = \begin{bmatrix}
      0 & 0 & 0 \\
      0 & 0 & 1 \\
      0 & 1 & 1
    \end{bmatrix}, \;
    D = \begin{bmatrix}
      0 & 2 & 0 \\
      0 & 0 & 0 \\
      0 & 0 & 0
    \end{bmatrix}, \;
    \vb{x}=( 2, 2, 3).
  \]
  Note the undirected edges of $\MG A{\vb{x}}$ are specified by $U$, and the directed ones by $D$.
    \begin{figure}[ht]
    \centering
    \includegraphics[width=6cm]{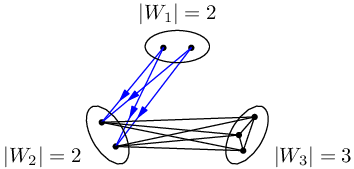}
    \caption{Example mixed matrix graph $\MG A{\vb{x}}$}
    \label{fig.mat-graph-ex}
    \end{figure}
\end{example}

For mixed adjacency matrix $A$ and mixed graph $F$, we say that $A$ is \textit{$F$-free} if the sequence $\{\MG{A}{t\vb 1}\}_{t\in\ZZ_{>0}}$ is $F$-free.
Given mixed adjacency matrix $A = (U, D)$, we also consider a single weighted adjacency matrix defined below that will enable us to approximate the weighted edge count of the associated mixed matrix graphs.

\begin{definition}\label{def.weighted}
  Let $A = (U,D)$ be a mixed adjacency matrix and fix $\rho\in(1,\infty)$. Let the associated \emph{weighted adjacency matrix} be $A_\rho \coloneqq U+\rho D$ with \emph{symmetric part} $\sym{A}{\rho} \coloneqq \frac12(A_\rho+A_\rho^\intercal)$.
\end{definition}

We denote the $(r-1)$-dimensional simplex in $\mathbb R^r$ by $\simplex r$, with
    \[
        \simplex r = \left\{\vb x = (x_1,\dots,x_r) \in \mathbb R^r \mid  \norm{\vb x}_1 = 1,  x_i \geq 0 \text{ for all } i\right\}.
    \]
    
We can formally associate a density to a mixed adjacency matrix.

\begin{definition}[Density]\label{def.mat-density}
  For a mixed adjacency matrix $A$ of size $r$ and fixed $\rho\in (1,\infty)$,
  let $$g_\rho(A) := \max_{\vb{y}\in \simplex{r}} \vb{y}^\intercal A_\rho\vb{y},$$
  which exists by compactness of $\simplex{r}$.
  We call $g_\rho(A)$ the \emph{$\rho$-density of $A$}.
\end{definition}

Towards our goal of constructing $F$-free mixed graphs with maximal $\rho$-density, we seek an ``asymptotically extremal mixed adjacency matrix.'' We first need to identify the relevant properties a mixed adjacency matrix should enjoy in order for it to support a sequence of dense $F$-free mixed graphs. This motivates the following notion, analogous to the ``dense'' matrices in~\cite{brown_erdos_simonovits} and ``minimal'' patterns in~\cite{Pikhurko_possible_densities}.

\begin{definition}[Principal submatrix]\label{def.principal-sub}
  Let $A = (U,D)$ and $A' = (U',D')$ be mixed adjacency matrices; $A'$ is a \emph{principal submatrix} of $A$, denoted $A' \trianglelefteq A$, if
  $U',D'$ are square matrices obtained 
  by removing the same set of corresponding rows and columns from both $U$ and $D$. We call $A'$ a \emph{proper principal submatrix} of $A$ if $A' \neq A$, denoted $A' \vartriangleleft A.$
\end{definition}
\begin{definition}[Condensed mixed adjacency matrix]\label{def.dense-mat}
  For $\rho\in(1,\infty)$, mixed adjacency matrix $A$ is \emph{condensed with respect to $\rho$} if $g_\rho(A')<g_\rho(A)$ for all $A' \vartriangleleft A$.
  If $A$ is a mixed adjacency matrix, we write $A\densesub A'$ when $A' \trianglelefteq A$ of smallest size such that $g_\rho(A')=g_\rho(A)$. Such $A'$ is condensed by definition.
\end{definition}
As the name suggests, for mixed adjacency matrices that are not condensed, we can further compress them by deleting (matching) rows and columns.
We now construct a finite collection of mixed adjacency matrices $\mathcal M_F$ that we will show in~\Cref{thm.matrix-maximal} contains an ``optimal'' $F$-free matrix (in the sense of yielding a sequence of mixed graphs that asymptotically maximizes the $\rho$-density).
Let $K$ be the size $1$ mixed adjacency matrix  $([1],[0])$; observe $\MG{K}{t\vb1} = K_t$ for all $t\in\ZZ_{>0}$.

\begin{definition}\label{def.mf-no-rho}
  For collapsible mixed graph $F$ where $\theta(F)\in(1,\infty)$, let $\mathcal M_F$ be the union of $\{K\}$ and the set of mixed adjacency matrices $A=(U,D)$ that satisfy the following conditions: 
  \begin{enumerate}[label=(\roman*)]
    \item $F\not\subseteq \MG{A}{t\vb1}$ for all $t \in \ZZ_{>0}$, i.e., $A$ is \textit{$F$-free};
    \item $\text{size}(A) \le \undirectedchi{F^\rhd}-1$;
    \item $U_{ii}=0$ for all $i \in [\text{size}(A)]$;
    \item $D$ is not identically zero;
    \item $U_{ij}+U_{ji}+D_{ij}+D_{ji}>0$ for all $i\neq j$.
  \end{enumerate}
\end{definition}
Observe that $\mathcal M_F$ is a finite set; for each $A \in \mcM_F$, we have $\text{size}(A) \le \upchi(F^{\rhd}) - 1$, and there are at most $3^{(r-1)r/2}$ elements of $\mathcal M_F$ of size $r$ (for each $1\leq i<j\leq r$ we have $(U_{ij},U_{ji},D_{ij},D_{ji}) \in \{(1,1,0,0)$, $(0,0,2,0),  (0,0,0,2)\}$).

With the above notation, we are prepared to state our main results. The proofs follow a similar structure to~\cite{brown_erdos_simonovits} and thus we defer the arguments to~\Cref{sec.matrices}.
We first give an \textit{approximation result} that enables us to describe ``maximal'' $F$-free mixed graphs as arising from blowups of some $A \in \mcM_F$ by a unique \textit{optimal vector}. 

\begin{lemma}[Convergence]\label{lem.convergence}
  Fix $\rho\in(1,\infty)$ and size $r$ condensed mixed adjacency matrix $A$. We have the following:
  \begin{enumerate}[label=(\roman*)]
    \item there exists a unique $\vb{y}^* \in \RR_{>0}^r$ such that for any sequence $\left\{\optvecn{A}{\rho}{n}\right\}_{n \in \ZZ_{>0}}$ where each $\optvecn{A}{\rho}{n} \in \ZZ_{\geq0}^r$ is some maximizer of $\ecw{\rho}{\MG{A}{\vb x}}$ over nonnegative $r$-vectors with $\norm{\vb x}_1 = n$, we have that
    $$\lim_{n\to\infty}\frac 1n \optvecn{A}{\rho}{n} = \vb{y}^*;$$
    \item $\vb{y}^*$ is the unique solution to
          \begin{equation}\label{eq.aaug}
            \vb y \in \simplex{r}
            ,\quad
            (\sym{A}{\rho}) \vb{y} = g_\rho(A)\vb1;
          \end{equation}
    \item $\vb{y}^*$ is the unique maximizer of $\vb{y}^\intercal A_\rho \vb{y}$
          among all  $\vb{y}\in\simplex{r}$.
  \end{enumerate}
\end{lemma}
The above convergence result identifies the critical properties a maximal blowup of a dense mixed adjacency matrix has, by constructing such a blowup, exhibiting an associated uniqueness property, and demonstrating that the $g_{\rho}(A)$ has a clean algebraic relationship to this blowup.

We are finally ready to precisely associate to each $F$ a matrix in $\mcM_F$ that yields a sequence of asymptotically extremal $F$-free mixed graphs; we defer the proof to~\Cref{sec.matrices}.
\begin{theorem}\label{thm.matrix-maximal}
Fix $\rho \in (1, \infty)$ and let $F$ be a mixed graph with at least one directed edge and $\theta(F) \in(1,\infty)$. Then,
\[
    \limsup_{n\to\infty} \max_{\substack{v(G)=n\\ F\not\subseteq G}} \frac{\ecw{\rho}{G}}{\tbinom{n}{2}} = \max_{B \in\mathcal M_F} g_\rho(B).
\]
\end{theorem}

\subsection{The variational characterization}\label{subsec.variational}
We will now use \Cref{lem.convergence} and \Cref{thm.matrix-maximal} to give a variational characterization of $\theta(F)$.

\begin{obs}\label{lemma.g-rho-continuous}
Fix mixed adjacency matrix $A$. Then $g_{\rho}(A)$ is continuous in $\rho$ for $\rho \in (1,\infty)$.
\end{obs}
\begin{proof}
    Let $A = (U,D)$. Since $\vb y^\intercal D \vb y \leq 1$ for any $\vb y \in \simplex{r}$, for all $\eps>0$ we have
    \[
    g_\rho(A)
    \leq g_{\rho+\eps}(A)
    = \max_{\vb y\in\simplex{r}} \vb y^\intercal (U+(\rho+\eps)D)\vb y 
    \leq \max_{\vb y\in\simplex{r}} \left(\vb y^\intercal (U+\rho D)\vb y + \eps\right)
    = g_\rho(A) + \eps,
    \]
    and similarly 
    $g_\rho(A) \geq g_{\rho-\eps}(A) \geq g_\rho(A) - \eps.$
    Continuity follows directly.
\end{proof}

We are now ready to give a formal variational characterization of $\theta(F).$

\begin{theorem}\label{thm.theta-variational}
  Let $F$ be a mixed graph with at least one directed edge such that $\theta(F) \in(1,\infty)$. For each  $B=(U,D) \in \mathcal M_F\setminus \{K\}$, let $r$ be its size. Then,
  \[\theta(F) = \min_{B\in\mathcal M_F\setminus \{K\}} \left\{\min_{\vb y \in \simplex{r}}\left\{\frac{1-\vb y^\intercal U \vb y}{\vb y^\intercal D \vb y}\right\} \right\},\]
  where we consider $\frac{1-\vb y^\intercal U \vb y}{\vb y^\intercal D \vb y}$ to be $\infty$ if $\vb y^\intercal D \vb y = 0$.
  Note that this implies $g_{\theta(F)}(B^*) = 1$, where $B^*$ is the mixed adjacency matrix attaining the above minimum.
\end{theorem}

\begin{proof}
    Consider any $B \in \mathcal M_F\setminus \{K\}$ of size $r$ and $\vb{y} \in \simplex{r}$. We first claim that
    $$\theta(F) \leq \frac{1-\vb y^\intercal U \vb y}{\vb y^\intercal D \vb y}.$$
    If not, we can find some mixed adjacency matrix $B$ and $\vb y \in \simplex{r}$ where $\vb y^\intercal U \vb y + \theta(F) \vb y^\intercal D \vb y > 1$. This would imply that $g_{\theta(F)}(B) > 1$, contradicting \eqref{eq.theta_equiv} by~\Cref{thm.matrix-maximal}.
    
    It remains to exhibit $B = (U, D) \in \mcM_F \backslash \{K\}$ and $y \in \simplex{r}$ achieving the aforementioned upper bound. By applying~\Cref{thm.matrix-maximal} to $\rho_k = \theta(F) + \frac1k$ for each $k \in \ZZ_{>0}$, we find for each $k$ an associated $B_k\in \mathcal M_F$ where 
    
    \begin{equation}\label{eq.rho-k-b-k}
        \limsup_{n\to\infty} \max_{\substack{v(G)=n\\ F\not\subseteq G}}
        \alpha(G) + \rho_k \beta(G) = g_{\rho_k}(B_k).   
    \end{equation}

    Since $\rho_k > \theta(F)$, the left hand side of~\eqref{eq.rho-k-b-k} is more than $1$. Therefore, $g_{\rho_k}(B_k)>1$ and thus $B_k \neq K$.
    Since $\mathcal M_F$ is finite, there must exist some matrix $B^* \in \mathcal M_F\setminus \{K\}$ and an infinite sequence $k_1 < k_2 < \cdots$ such that $B^*= B_{k_1} = B_{k_2} = \cdots$. Recalling \Cref{lemma.g-rho-continuous}, we conclude that
    \[g_{\theta(F)}(B^*) = \lim_{k\to\infty} g_{\rho_k}(B^*) \geq 1.\]
    However, by \Cref{thm.matrix-maximal}, $g_{\theta(F)}(B^*)\leq 1$, so $g_{\theta(F)}(B^*) = 1$ and the conclusion follows.
\end{proof}

%%%%%%%%%%%%%%%%%%%%%%%%%%%%%%%%%
%%   ALGEBRAIC density coeff   %%
%%%%%%%%%%%%%%%%%%%%%%%%%%%%%%%%%

\section{Algebraic degree of \thetaF}\label{sec.algebraic-degree}

Thus far, the explicit results we have obtained bears striking resemblance to classical extremal graph theory. For instance,~\Cref{mixedgraph-general} shows that the Tur\'an density coefficients of a mixed graph with one directed edge is exactly characterized by the Tur\'an density of its underlying undirected graph, and that in general this is not far from true numerically. However, when we think about the space of \textit{values} the Tur\'an density coefficient can take on, the analogy to undirected graphs will turn out to break down in a striking manner. 

Tur\'an densities of undirected graphs $H$ with $\chi(H) \ge 3$ are always of the form $1 - 1/(\chi(H)-1)$ (and thus $1 - 1/k$ for some $k \in \ZZ_{>0}$) by the Erd\H{o}s-Stone-Simonovits theorem. In contrast, we show by example that $\theta(F)$ is not even always rational (in~\Cref{subsec.irrational-theta}). While we subsequently observe in~\Cref{subsec.algebraic-theta} that Tur\'an density coefficients are always algebraic, this is in some sense the ``strongest'' result we could hope for; we prove~\Cref{thm.alg-degree}, an analogue of a corresponding hypergraph Tur\'an result of~\cite{pikhurko-alg-degree}, showing that there are finite families of mixed graphs $\mcF$ such that $\theta(\mcF)$ has arbitrarily high algebraic degree.

\subsection{A mixed graph \F{} with irrational \thetaF}\label{subsec.irrational-theta}
We begin with a simple, explicit example of a mixed graph with irrational Tur\'an density coefficient. This example also highlights the power of~\Cref{thm.theta-variational}, which enables us to exactly compute the Tur\'an density coefficient of any given mixed graph by solving a finite number of relatively small quadratic fractional programs over the simplex (with each program of dimension bounded by $\upchi(F^{\rhd}) - 1$).

\begin{proposition}\label{thm.irrational-theta}
The mixed graph pictured in~\Cref{fig.irrational-theta} has Tur\'an density coefficient $1+\frac1{\sqrt2}$.

\begin{figure}[ht!]
  \centering
  \includegraphics[width=11cm]{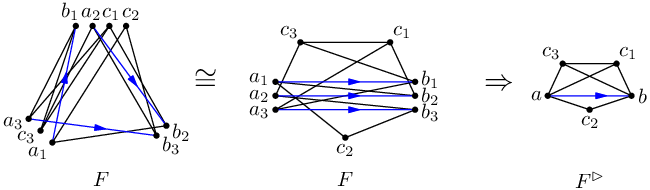}
  \caption{The above mixed graph $F$ has $\theta(F) = 1+\frac1{\sqrt2}$.}
  \label{fig.irrational-theta}
\end{figure}
\end{proposition}
\begin{proof}
By inspection of~\Cref{fig.irrational-theta}, we have $\undirectedchi{F} = 3$ and $\undirectedchi{F^\rhd}=4$. 
We consider all mixed graphs with at least one directed edge and at most 3 vertices, as shown in \Cref{fig.exg1234}.
\Cref{fig.fsubg3t} gives an embedding for $F\subseteq G_3[3]$ which also implies that $F\subseteq G_6[3]$;
\Cref{fig.fsubg5t}  gives an embedding for $F\subseteq G_5[3]$;
and~\Cref{fig.fsubg7t} gives an embedding for $F\subseteq G_7[3]$.

\begin{figure}[ht]
  \centering
  \includegraphics[width=12cm]{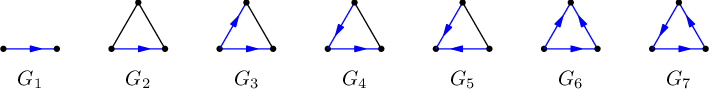}
  \caption{Mixed graphs with at most 3 vertices and at least 1 directed edge}
  \label{fig.exg1234}
\end{figure}
    
\begin{figure}[ht]
  \centering
  \begin{minipage}{.25\textwidth}
    \centering
    \includegraphics[width=3.5cm]{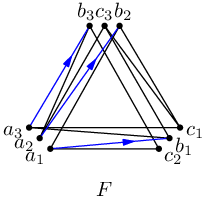}
    \captionsetup{width=0.5\linewidth}
    \caption{$F\subseteq G_3[3]$}
    \label{fig.fsubg3t}
  \end{minipage}%
  \begin{minipage}{.25\textwidth}
    \centering
    \includegraphics[width=3.5cm]{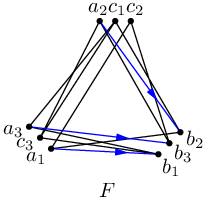}
    \captionsetup{width=0.5\linewidth}
    \caption{$F\subseteq G_5[3]$}
    \label{fig.fsubg5t}
  \end{minipage}%
  \begin{minipage}{.25\textwidth}
    \centering
    \includegraphics[width=3.5cm]{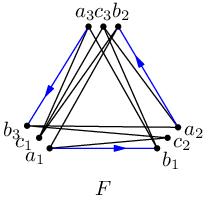}
    \captionsetup{width=0.5\linewidth}
    \caption{$F\subseteq G_7[3]$}
    \label{fig.fsubg7t}
  \end{minipage}%
  \begin{minipage}{.25\textwidth}
    \centering
    \includegraphics[width=3.5cm]{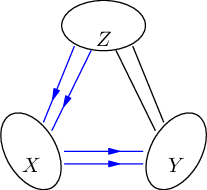}
    \captionsetup{width=0.5\linewidth}
    \caption{$F\not\subseteq G_4[t]$}
    \label{fig.fsubg4t}
  \end{minipage}
\end{figure}
    However, we will show $F\not\subseteq G_4[t]$, which also implies that $F\not\subseteq G_1[t]$ and $F\not\subseteq G_2[t]$, for all $t \in \ZZ_{>0}$.
    Suppose to the contrary there was some $t$ and let $V(G_4[t]) =  X \sqcup Y \sqcup Z$ with each vertex part the $t$-blowup of its corresponding vertex of $G_4$ (see \Cref{fig.fsubg4t}). Consider an embedding of $F$ into $G_4[t]$ with vertices labelled as in~\Cref{fig.irrational-theta}.
    Since $a_1$ is the tail vertex of $a_1\widecheck{b_1} \in E(F)$, we have $a_1 \in X$ or $a_1 \in Z$. We can show in both cases, by first considering which part $b_1$ must fall in and other subsequently forced vertices, that we must embed $c_1, c_3$ into the same part of $G_4[t]$, which yields a contradiction since $c_1 c_3 \in E(F)$.
    
    We can then compute $\theta(F)$ by applying~\Cref{thm.theta-variational}. We consider all $F$-free mixed adjacency matrices (up to isomorphism) with size at most $\undirectedchi{F^\rhd} - 1 = 3$, enumerated in the below table. We also include in the table the associated minimizing vector $\vb x^*$ of $\min_{\vb x\in\simplex{r}}\frac{1-\vb x^\intercal U \vb x}{\vb x^\intercal D \vb x}$.

    \begin{center}
    \begin{tabular}{c|c|c|c}
      $\MG{B}{\vb1}$              & $B=(U,D)$                                                                                                                                                        & $\vb x^*$                                                 & $\displaystyle{\min_{\vb x\in\simplex{r}}\frac{1-\vb x^\intercal U \vb x}{\vb x^\intercal D \vb x}}$ \\ \hline
      \rule{0pt}{8mm}
      $G_1$                       & $\displaystyle{\left(\begin{bmatrix}0&0\\0&0\end{bmatrix},  \begin{bmatrix}0&2\\0&0\end{bmatrix}\right)}$                                                        & $(\frac12,\frac12)$                                       & $2$
         \\ 
      \rule{0pt}{11mm}
      $G_2$                       & $\displaystyle{\left(\begin{bmatrix}0 & 0 & 1 \\ 0 & 0 & 1\\ 1 & 1 & 0 \end{bmatrix}  ,\begin{bmatrix}0 & 2 & 0 \\ 0 & 0 & 0 \\ 0 & 0 & 0 \end{bmatrix}\right)}$ & $( \frac12,\frac12,0 )$                                   & $2$
         \\
      \rule{0pt}{11mm}      $G_4$ & $\displaystyle{\left(\begin{bmatrix}0&0&1\\0&0&0\\1&0&0\end{bmatrix}, \begin{bmatrix} 0 & 2 & 0 \\ 0 & 0 & 2 \\ 0 & 0 & 0 \end{bmatrix}\right)}$                 & $\left(1-\frac1{\sqrt2},\sqrt2-1,1-\frac1{\sqrt2}\right)$ & $1+\frac1{\sqrt2}$
    \end{tabular}
    \end{center}
    Hence, we have as desired that 
    \[\theta(F) = \min_{(U,D)\in\mathcal M_F\setminus \{K\}}\min_{\vb y\in \simplex{r}}\left\{\frac{1-\vb y^\intercal U \vb y}{\vb y^\intercal D \vb y}\right\} = 1+\frac1{\sqrt2}.\]
\end{proof}

\subsection{\thetaF{} is always algebraic}\label{subsec.algebraic-theta}
We now establish algebraicity of $\theta(F)$. This will be a consequence of the observation that $\theta(F)$ can be computed as a rational function of the inverse of a finite matrix with rational coefficients.

\begin{proof}[Proof of \Cref{thm.mixedgraph-algebraic}]
    Let $B^*$ be the minimal mixed adjacency matrix in $\mathcal M_F\setminus\{K\}$ (described in \Cref{thm.theta-variational}), so $g_{\theta(F)}(B^*)=1$.
    Let  $\rho=\theta(F)$ and take $B^*\densesub B$ (recall \Cref{def.dense-mat}) with respect to $\rho$.
    Critically, $g_{\rho}(B) = 1$.
    Let $r$ be the size of $B$. \Cref{lem.convergence} guarantees that there exists a vector $\vb y^* \in \RR_{>0}^r$ which maximizes $\vb y^\intercal \sym{B}{\rho} \vb y$ (recall \Cref{def.weighted}) subject to $\vb y \in \simplex{r}$. In addition, $\vb y^*$ is the unique solution to 
    \[\sym{B}{\rho}\vb{y} = \vb 1, \quad \vb y \in \simplex{r}\]
    since $g_\rho(B) = 1$.
    Now, define the $(r+1)\times r$ matrix
    \[C = \begin{bmatrix} \sym{B}{\rho} \\ \vb 1^{\intercal} \end{bmatrix}.\]
    By the above, $\vb y^*$ is the unique solution to $C\vb y = \vb 1$ in $\simplex{r}$.
    Note that any solution $\vb y$ of $C\vb y = \vb 1$ must satisfy $\norm{\vb y}_1 = 1$.
    Since $\vb y^* \in \RR_{>0}^r$, it is possible to take a neighborhood of $\vb y^*$ in $\RR^r$ which is wholly contained in $\simplex{r}$, so $C\vb y = \vb 1$ has unique solution $\vb y^*$ in the neighborhood.
    This implies $\vb y^*$ is in fact the only solution to $C\vb y = \vb 1$ in $\RR^r$.
    Hence $C$ is a rank $r$ matrix, which means there is some row we can delete to obtain invertible $A \in \RR^{r \times r}$.
    Observe that $\vb y^*$ is the unique solution to $A\vb y = \vb 1$ in $\RR^r$. Consequently, $\vb y^* = A^{-1} \vb 1$.
    Since all entries of $A$ lie in $\{0, 1, \rho\}$, we can express all entries of $A^{-1}$ as $\frac{P(\rho)}{Q(\rho)}$ where $P,Q$ are polynomials in $\ZZ[\rho]$. This gives a corresponding expression for $\vb{y}$, the sum of the rows of $A^{-1}$, as a quotient of integer polynomials of $\rho$.
    By \Cref{thm.theta-variational} we have
    \[\rho =\theta(F) = \max_{\vb y\in\simplex{r}} \frac{1-\vb{y}^\intercal U \vb{y}}{\vb{y}^\intercal D \vb{y}} = \frac{1-\vb{y^*}^\intercal U \vb{y^*}}{\vb{y^*}^\intercal D \vb{y^*}}.\]
    After clearing the denominator, we obtain a polynomial with integer coefficients where $\rho =\theta(F)$ is a root, implying $\theta(F)$ is algebraic as claimed.
\end{proof}

\subsection{Finite families of mixed graphs with arbitrarily high algebraic degree}\label{subsec.high-degree}
We prove \Cref{thm.alg-degree}, an analogue of the main result of~\cite{pikhurko-alg-degree}, where it was observed that there are families of hypergraphs with Tur\'an densities having arbitrarily high algebraic degree. Precisely, we show that for each $k \in \ZZ_{>0}$, there is a finite family of mixed graphs $\mcF$ such that $\theta(\mcF)$ has algebraic degree $k$.

For a finite family $\mcF$ of mixed graphs, we say mixed graph $G$ is \textit{$\mcF$-free} if $F\not\subseteq G$ for all $F \in \mcF$; the Tur\'an density coefficient $\theta(\mcF)$ is defined analogously as the maximum $\rho$ such that 
  \[\alpha(G) + \rho \cdot  \beta(G) \leq 1+o(1)\]
over all $\mcF$-free $n$-vertex mixed graphs $G$. If $\beta(G)=o(1)$ over $\mcF$-free mixed graphs $G$, we say $\theta(\mcF)=\infty$. This is a well-defined quantity by the same reasoning as \Cref{prop.theta-exists}, and the results of~\Cref{sec.variational} all generalize verbatim to finite forbidden families such that at least one member of the family has at least one directed edge.

\begin{lemma}\label{l.find-family}
    Let $B$ = $(U,D)$ be a mixed adjacency matrix of size $r$ satisfying $D \neq \mathbf{0}$ and $U_{ii} = 0$ for all $i \in [r]$, such that the mixed matrix graph $\MG{B}{\vb1}$ has complete underlying undirected graph.
    Then, there exists a finite family $\mcF$ of mixed graphs such that 
    \[\theta(\mcF) = \min_{\vb y\in\simplex{r}} \left\{\frac{1-\vb{y}^\intercal U \vb y}{\vb{y}^\intercal D \vb y}\right\}.
    \]
\end{lemma}
\begin{proof}
    We let $\mcF$ be the family of mixed graphs $F$ such that $v(F) \le r+1$ and $F \not \subseteq \MG{B}{t\vb1}$ for all $t \in \ZZ_{>0}$ (so that $B$ is \textit{$F$-free}).
    We first check that $\theta(\mcF) \in (1, \infty)$. Since $D$ is not the zero matrix, $K_{\overrightarrow{t,t}} \subseteq \MG{B}{t\vb1}$. Consequently, $K_{\overrightarrow{t,t}}$ is $\mcF$-free for all $t \in \ZZ_{>0}$. 
    This implies that $\theta(\mcF) \le 2$ by \Cref{prop.ktt}.
    On the other hand, observe that $\overrightarrow{K_{r+1}} \in \mcF$; hence, $\theta(\mcF) \geq \theta(\overrightarrow{K_{r+1}}) = 1+\frac1{r-1}>1.$
    
    Hence we may apply the variational characterization of~\Cref{thm.theta-variational} (for family $\mcF$), deducing that there exists $B' = (U', D')\in\mathcal M_\mcF\setminus\{K\}$ of size $r'$ such that $\theta(\mcF) = \min_{\vb y\in\simplex{r'}}\frac{1-\vb{y}^\intercal U' \vb y}{\vb{y}^\intercal D' \vb y}$. By~\Cref{def.mf-no-rho}, $B'$ is $\mcF$-free, $\MG{B'}{\vb1}$ has complete underlying undirected graph, $D' \neq \mathbf{0}$, and $U'_{ii} = 0$ for all $i$; furthermore $r'\leq r$ since $\overrightarrow{K_{r+1}}\subseteq\mcF$.
    Since $\MG{B'}{\vb 1}$ is $\mcF$-free with $r' \le r < r+1$ vertices, for some $t \in \ZZ_{>0}$ it is true that $\MG{B'}{\vb 1} \subseteq \MG{B}{t\vb 1}$. Since both $B$ and $B'$ have complete underlying undirected graphs, in fact, we must have $\MG{B'}{\vb 1} \subseteq \MG{B}{\vb 1}$, i.e., $B' \trianglelefteq B$, implying that
    \[\theta(\mcF) = \min_{\vb y\in\simplex{r'}}\left\{\frac{1-\vb{y}^\intercal U' \vb{y}}{\vb{y}^\intercal D' \vb{y}}\right\} \geq \min_{\vb{y}\in\simplex{r}} \left\{\frac{1-\vb{y}^\intercal U \vb{y}}{\vb{y}^\intercal D \vb{y}}\right\}.\]
    The reverse inequality follows since $B$ is $\mcF$-free by definition of $\theta(\mcF)$.
\end{proof}

Define a sequence of mixed adjacency matrices $\{B_k = (U_k, D_k)\}_{k \in \ZZ_{\geq0}}$ by $B_0 = ([0],[0])$ and
\[
    B_{k+1} =
    \left(\begin{bmatrix}
    {\tiny \begin{matrix}0&0\\0&0\end{matrix}}                & {\tiny \begin{matrix}0&\cdots&0\\1&\cdots&1\end{matrix}} \\
    {\tiny \begin{matrix}0&1\\\vdots&\vdots\\0&1\end{matrix}} & {\Huge U_k}
    \end{bmatrix},
    \begin{bmatrix}
    {\tiny \begin{matrix}0&2\\0&0\end{matrix}} & {\tiny\begin{matrix} 2 & \cdots & 2\\0 & \cdots & 0\end{matrix}} \\ {\tiny \begin{matrix}0 & 0 \\ \vdots & \vdots \\ 0 & 0 \end{matrix}} & {\Huge D_k}
    \end{bmatrix}\right)
\]
so that $B_k$ has size $2k+1$. For each $k \in \ZZ_{\geq0}$, let $\mcF_k$ be the finite family given by \Cref{l.find-family} for $B_k$, so
\[\theta(\mcF_k) = \min_{\vb{y}\in\simplexx{2k}}\left\{\frac{1-\vb{y}^\intercal U_k \vb{y}}{\vb{y}^\intercal D_k \vb{y}}\right\}.\]
Since $B_k \vartriangleleft B_{k+1}$, we have $\theta(\mcF_k) \geq \theta(\mcF_{k+1})$.

\begin{definition}
    For mixed adjacency matrix $B$ of size $r$ and $\rho\in(1,\infty)$, call vector $\vb{y}^* \in\simplex{r}$ a \emph{maximizing vector of $B$ with respect to $\rho$} if
    \[{\vb{y}^*}^\intercal B_\rho \vb{y}^* = g_\rho(B) = \max_{\vb{y}\in\simplex{r}} \vb{y}^\intercal B_\rho \vb{y}.\]
\end{definition}

\begin{proposition}
    Suppose $\vb{y}^* = (y^*_1, \dots, y^*_{2k+1})$ is a maximizing vector of $B_k$ with respect to some $\rho \in \big(1, \theta(\mcF_k)\big]$.
    Then, the vector
    \[
    \left(u^*,v^*,(1-u^*-v^*)y^*_1, \dots, (1-u^*-v^*)y^*_{2k+1}\right),
    \]
    where $u^*=\frac{\rho(2-g_\rho(B_k))-1}{4\rho-2\rho g_\rho(B_k)-1}$
    and $v^*=\frac{\rho(1-g_\rho(B_k))}{4\rho-2\rho g_\rho(B_k)-1}$,
    is maximizing for $B_{k+1}$ with respect to $\rho$.
\end{proposition}
\begin{proof}
    Let $\vb{y} = (u, v, y_1, \dots, y_{2k+1})$ be a vector of length $2k+3$; and let $\vb{y}' = \frac1{1-u-v} (y_1, \dots, y_{2k+1})$. Then we have
    \[\vb{y}^\intercal (B_{k+1})_\rho \vb{y} = 2\rho u(1-u) + 2v(1-u-v) + (1-u-v)^2 {\vb{y}'}^\intercal (B_k)_\rho \vb{y}'.\]
    From above, we see there exists some maximizing vector of $B_{k+1}$ with respect to $\rho$ (that we term $\widehat{\vb{y}} = (u,v,\widehat y_1, \dots, \widehat y_{2k+1})$) such that $(\widehat y_1, \dots, \widehat y_{2k+1}) = (1-u-v) \vb{y}^*$.
    We define
    \[f(u, v) := \widehat{\vb{y}}^\intercal (B_{k+1})_\rho \widehat{\vb{y}} = 2\rho u(1-u) + 2v(1-v) - 2uv + (1-u-v)^2 g_\rho(B_k).\]
    We then compute the Hessian of $f(u, v)$ for $(u, v) \in \triangle^2$ as
    \[\nabla^2 f(u, v) = \begin{bmatrix}
    2g_\rho(B_k) - 4\rho & 2g_\rho(B_k)-2 \\ 2g_\rho(B_k)-2 & 2g_\rho(B_k) - 4
    \end{bmatrix}.\]
    Since $\rho \in (1, \theta(\mcF_k)]$ we have that $g_{\rho}(B_k) \in [0, 1]$. Thus for any vector $\vb x = (x_1, x_2) \in \RR^2\setminus\{(0,0)\}$,
    \[
    \vb x^\intercal \left(\nabla^2 f(u, v)\right) \vb x = -\left((4\rho-2)x_1^2 + 2x_2^2 + (2-2g_\rho(B_k))(x_1+x_2)^2\right) < 0.
    \]
    Therefore, $\nabla^2 f(u, v)$ is negative definite, implying that
    $f(u, v)$ is concave over $\triangle^2$.
    At $u=u^*$ and $v=v^*$ (for $u^*, v^*$ as given above which satisfy $(u^*, v^*) \in \triangle^2$),
    we have $$\frac{\partial\big(\widehat{\vb{y}}^\intercal (B_{k+1})_\rho \widehat{\vb{y}}\big)}{\partial u} = \frac{\partial\big(\widehat{\vb{y}}^\intercal (B_{k+1})_\rho \widehat{\vb{y}}\big)}{\partial v} = 0,$$
    hence these values of maximize $f(u, v)$ and the conclusion follows.
\end{proof}
\begin{corollary}\label{prop.alg-recur-g}
    For $\rho\in \big(1,\theta(\mcF_{k})\big]$,
    \[g_\rho(B_{k+1}) = \frac{\rho^2 (2-g_\rho(B_k))}{2\rho (2-g_\rho(B_k)) - 1}.\]
\end{corollary}

We are now ready to prove \Cref{thm.alg-degree}.
\begin{proof}[Proof of \Cref{thm.alg-degree}]
    For all $k\geq0$ define the polynomials $P_k(\rho), Q_k(\rho)$ recursively by
    \begin{align*}
        P_0(\rho) = 0;                                & \quad Q_0(\rho) = 1;                                              \\
        P_{k+1}(\rho) = \rho^2(2Q_k(\rho)-P_k(\rho)); & \quad Q_{k+1}(\rho) = (4 \rho - 1) Q_k(\rho) -  2 \rho P_k(\rho).
    \end{align*}
    By induction on $k$, we can observe the following for all $k \in \ZZ_{>0}$:
    \begin{enumerate}[label=(\roman*)]
        \item $g_\rho(B_{k}) = \frac{P_{k}(\rho)}{Q_{k}(\rho)}$ for all $\rho \leq \theta(\mcF_{k-1})$ (by \Cref{prop.alg-recur-g});
        \item $P_k$ has degree $2k$, and $Q_k$ has degree $2k-1$;
        \item $Q_k(0) = \pm 1$;
        \item $\rho^2 \mid P_k$, $P_k$ has all even coefficients, and its leading coefficient is $\pm2$.
    \end{enumerate}
    Hence for $k\geq1$ we know
    \[P_{k}(\rho)-Q_{k}(\rho) = \big(\rho^2(2Q_{k-1}(\rho)-P_{k-1}(\rho))\big) - \big(2\rho(2Q_{k-1}(\rho)-P_{k-1}(\rho)) - Q_{k-1}(\rho)\big)\]
    is a polynomial of degree $2k$, with constant term $\pm1$, all other coefficients even, and leading coefficient $\pm2$. Hence the reciprocal polynomial of $P_{k} - Q_{k}$ (i.e., the polynomial with the coefficients of $P_{k} - Q_{k}$ in reverse order) is irreducible over $\ZZ[\rho]$ by Eisenstein's criterion on the prime $2$, which implies $P_{k} - Q_{k}$ itself is irreducible as well.
    
    Finally, since $\theta(\mcF_{k}) \leq \theta(\mcF_{k-1})$ and $g_{\theta(\mcF_{k})}(B_{k}) = 1$, we have by condition (1) that  $P_{k}(\rho) = Q_{k}(\rho)$. In particular, $\theta(\mcF_{k})$ must be a root of $P_{k}(\rho) - Q_{k}(\rho)$, which is an irreducible polynomial of degree $2k$ with integer coefficients. Hence $\theta(\mcF_{k})$ has algebraic degree $2k$, implying that for every even algebraic degree $2k$ for $k \in \ZZ_{>0}$, we have some family $\mcF_k$ with $\theta(\mcF_k)$ having algebraic degree $2k$ (\Cref{thm.irrational-theta} gives an example mixed graph $F$ where $\theta(F)$ has algebraic degree $2$).

    For odd algebraic degrees, let $B_k'$ be the  principal submatrix obtained by deleting the first row/column from the matrices in $B_k$; define $\mcF_k'$ for $B_k'$ in the same fashion as $\mcF_k$ is defined for $B_k$. By repeating essentially the same argument as above, we find that $\theta(\mcF_k')$ has algebraic degree $2k-1$ for all $k\geq1$, which means all odd algebraic degrees are attainable.
\end{proof}

\section*{Acknowledgements}

This paper represents the results of a research project through the MIT PRIMES-USA program. The authors would like to thank the organizers of the PRIMES-USA program for providing this research opportunity and for helpful feedback. We would also like to thank Professor Yufei Zhao for helpful suggestions and feedback. Mani was supported by a Hertz Fellowship and the NSF Graduate Research Fellowship Program.

\clearpage
\appendix
\begin{appendices}

%%%%%%%%%%%%%%%%%%%%%%%%
%%   MIXED MATRICES   %%
%%%%%%%%%%%%%%%%%%%%%%%%

\section{Mixed adjacency matrices}\label{sec.matrices}
We generalize the main arguments of~\cite{brown_erdos_simonovits} to the setting of mixed graphs. The main ideas bear strong resemblance to those in~\cite{brown_erdos_simonovits}, but the notion of subgraphs in mixed graphs makes the analysis more difficult. Unless otherwise specified, in this section, we consider $\rho \in (1, \infty)$ to be a fixed constant. Also, we fix a mixed adjacency matrix $A$ of size $r$ in most of the section.
 
\begin{proposition}\label{prop.matrix-uau}
  Fix $\rho \in (1, \infty)$ and let $A$ be a size $r$ mixed adjacency matrix. For all $\vb x\in\ZZ_{\geq0}^r$, 
  \[\frac12 \vb{x}^\intercal A_\rho \vb{x} - \frac12 \norm{\vb x}_1
  \leq \ecw{\rho}{\MG{A}{\vb{x}}}
  \leq\frac12 \vb{x}^\intercal A_\rho \vb{x}.\]
\end{proposition}
\begin{proof}
This is similar to Equation (1) of \cite{brown_erdos_simonovits}.
Let $A = (U, D)$ and $\vb{x} = (x_1,\dots,x_r)$. Using notation from \Cref{def.mat-graph}, let $\MG{A}{\vb{x}}$ have components $W_1, \dots, W_r$.
The counts of undirected edges and directed edges  between vertex components $W_i$ and $W_j$ are $\frac12(U_{ij} + U_{ji}) x_i x_j$ and 
  $\frac12(D_{ij} + D_{ji}) x_i x_j$, respectively, and the number of undirected edges within component $W_i$ is
  $U_{ii}\binom{x_i}{2}$.
  Summing yields the desired result.
\end{proof}

We note a crude bound on the density of any mixed adjacency matrix that will be useful later; $g_\rho(A)<\rho$, as for each mixed adjacency matrix $A$, all terms of $\sym{A}{\rho}$ are at most $\rho$.

We will establish a link between the density $g_\rho(A)$ and the asymptotic maximum of $\ecw{\rho}{\MG{A}{\vb x}}$ as $\norm{\vb x}_1$ grows large. To do this, we will consider the following analogue of Definition 2 in \cite{brown_erdos_simonovits}:
\begin{definition}[Maximal mixed matrix graph]\label{def.max-mat}
  Let $A$ be a mixed adjacency matrix of size $r$, choose $\rho\in (1,\infty)$, and let $n \in \ZZ_{\geq0}$.
  \begin{enumerate}[label=(\roman*)]
    \item Let $\optvecn{A}{\rho}{n}$ maximize
    $\ecw{\rho}{\MG{A}{\vb x}}$
    over all $\vb x \in \ZZ_{\geq0}^r$ with $\norm{\vb{x}}_1 = n$.
    \item Let $\MMG{A}{\rho}{n} \coloneqq \MG{A}{\optvecn{A}{\rho}{n}}$ be the \emph{maximal mixed matrix graph} on $n$ vertices.
  \end{enumerate}
\end{definition}

In essence, $\MMG{A}{\rho}{n}$ is the $n$-vertex asymmetric blowup of $A$ with the largest weighted edge count of all such blowups.
By applying \Cref{def.max-mat} and \Cref{prop.matrix-uau} we can draw the following conclusion, similar to Equation 3 in~\cite{brown_erdos_simonovits}:
\begin{proposition}\label{prop.fin_inf_eq}
  Let $A$ be a mixed adjacency matrix and $\rho\in(1,\infty)$. For all positive integers $n$,
  \[\ecw{\rho}{\MMG{A}{\rho}{n}} =\frac{n^2}{2} g_\rho(A) + o\left(n^2\right).\]
\end{proposition}

\subsection{Convergence}
We first observe that maximal mixed graphs are roughly regular, in a suitable weighted sense via a Zykov symmetrization-style argument.
\begin{lemma}[Zykov symmetrization]\label{prop.mat-zykov}
    Let $A$ be a mixed adjacency matrix, $n$ be a positive integer, and $\rho\in(1,\infty)$. 
    Let the \emph{$\rho$-weighted degree} of a vertex $v \in V(G)$ be $\degw v := \degu v + \rho\degd v$; 
    Then $\abs{\degw v_1 - \degw v_2} \leq \rho$ for any vertices $v_1, v_2$ of $\MMG{A}{\rho}{n}$.
\end{lemma}
\begin{proof}
    Let $A=(U,D)$, and $\vb{x}^{(n)}_{\rho,A}$ be the vector from \Cref{def.max-mat}.
    By the construction given in \Cref{def.mat-graph}, if $v_1, v_2$ are in the same component $W_i$ then $\degw{v_1} = \degw{v_2}$. 
    Thus, assume to the contrary (without loss of generality) that there are $v_1\in W_1$ and $v_2\in W_2$ such that $\degw{v_1} > \degw{v_2} + \rho$.
    
    Consider the vector $\vb{x}'=\vb{x}^{(n)}_{\rho,A}+( 1, -1, 0, \dots, 0)$ (i.e., the mixed graph that results from moving vertex $v_2$ from $W_2$ to $W_1$). We have $\vb{x}'\in \ZZ_{\geq0}^r$ and $\norm{\vb{x}'}_1 = n$, but
    \[
    \ecw{\rho}{\MG{A}{\vb{x}'}} =
    \ecw{\rho}{\MG{A}{\vb{x}^{(n)}_{\rho,A}}} + \degw{v_1} - \degw{v_2}
    - \big(\sym{A}{\rho}\big)_{12} + U_{11}
    > \ecw{\rho}{\MG{A}{\vb{x}^{(n)}_{\rho,A}}},
    \]
    where the last inequality follows from 
    $\degw{v_1} - \degw{v_2} > \rho \geq \left(\sym{A}{\rho}\right)_{12}$ 
    and $U_{11} \geq 0$.
    This contradicts the maximality of $\vb{x}^{(n)}_{\rho,A}$.
\end{proof}

We will now prove the convergence lemma,
\Cref{lem.convergence}.
The bulk of the proof is in several claims. Roughly speaking, the first claim establishes a relation between \eqref{eq.aaug} and the argmax of $\vb{y}^\intercal A_\rho \vb{y}$, the second shows that $\lim_{n\to\infty}\frac1n \vb{x}_{\rho,A}^{(n)}$ exists, and the third shows that the limit has the desired properties.

\begin{claim}\label{claim.convergence1}
    Any vector $\vb{y}$ which maximizes $\vb{y}^\intercal A_\rho \vb{y}$ subject to $\vb y \in \simplex{r}$ satisfies \eqref{eq.aaug}.
\end{claim}
\begin{proof}
    This follows via Lagrange multipliers. Since
    \[\nabla \vb{y}^\intercal A_\rho \vb{y} = \left(A_\rho+A_\rho^\intercal\right)\vb{y} = 2\left(\sym{A}{\rho}\right)\vb{y}
    \quad\text{and}\quad \nabla \norm{\vb{y}}_1 = \vb{1},
    \]
    a maximizing vector $\vb{y}^*\in\simplex{r}$ of $\vb y^\intercal A_\rho \vb y$ either has some coordinate equal to zero (i.e., is on the boundary of $\simplex{r}$), or satisfies
    \[0 = \Big(\nabla \vb{y}^\intercal A_\rho \vb{y} -\lambda \nabla \norm{\vb{y}}_1\Big)\Big|_{\vb{y}=\vb{y}^*}
    = 2\left(\sym{A}{\rho}\right)\vb{y}^*-\lambda \vb{1}\]
    for some $\lambda \in \RR$.

    We dispose of the first case: suppose some coordinate of $\vb y^*$ were equal to zero; let $\vb y'$ be the $(r-1)$-dimensional vector with that coordinate removed from $\vb y^*$, and $A'$ be the principal submatrix with the row and column corresponding to that coordinate removed. Then we have $g_\rho(A') \geq {\vb y'}^\intercal A' \vb y' = {\vb y^*}^\intercal A' \vb y^* = g_\rho(A)$, which implies $g_\rho(A') = g_\rho(A)$, contradicting the fact that $A$ is condensed. 
    
    In the second case, we have
    \[\lambda (\vb{y}^*)^\intercal \vb1 = 2(\vb{y}^*)^\intercal \left(\sym{A}{\rho}\right) \vb{y}^* = 2 g_{\rho}(A)\]
    which implies $\lambda = \frac{2 g_{\rho}(A)}{(\vb{y}^*)^\intercal \vb1} = \frac{2 g_{\rho}(A)}{\norm{\vb{y}^*}_1} = 2 g_{\rho}(A)$; hence
    \[2\left(\sym{A}{\rho}\right)\vb{y}^* = \lambda \vb{1} = 2 g_{\rho}(A) \vb 1.\]
\end{proof}

\begin{claim}\label{claim.convergence2}
    If \eqref{eq.aaug} has a solution, then it is unique and has all coordinates strictly positive.
    \end{claim}
    \begin{proof}
    If $\vb{y}^*$ is a solution to \eqref{eq.aaug} then
    \[ (\vb{y}^*)^\intercal A_\rho \vb{y}^* 
    = (\vb{y}^*)^\intercal (\sym{A}{\rho}) \vb{y}^* 
    = (\vb{y}^*)^\intercal \cdot g_\rho(A)\vb1
    = g_\rho(A),\]
    the last equality being true since $\norm{\vb{y}^*}_1 = 1$.
    
    We first show that no solution can have coordinates that are $0$.
    Suppose to the contrary that \eqref{eq.aaug} has a solution $\vb{y}_0$ 
    with some coordinate equal to $0$.
    Let $\vb{y}'$ be the $(r-1)$-dimensional vector with that coordinate removed from 
    $\vb{y}_0$, and let $A'$ be the principal submatrix of $A$ with the row and column corresponding to that coordinate removed.
    Then $\left(\vb{y}'\right)^\intercal A'_\rho \vb{y}' 
    = \vb{y}_0^\intercal A_\rho \vb{y}_0 = g_\rho(A)$, 
    contradicting the fact that $A$ is condensed.
    
    Now we show that there cannot be more than one solution. Suppose for the sake of 
    contradiction that there are two distinct vectors with nonnegative coordinates
     which both satisfy \eqref{eq.aaug}, which means any affine combination of them does as well.
    A suitable combination produces a vector satisfying \eqref{eq.aaug}
    with all coordinates nonnegative and at least one coordinate $0$,
    reducing to the first case.
    \end{proof}
    
    \begin{claim}\label{claim.convergence3}
    For any limit point $\widehat{\vb{y}}=( y_1,\dots,y_r) \in \simplex{r}$ of $\left\{\frac1n \vb{x}_{\rho,A}^{(n)}\right\}_{n \in \ZZ_{>0}}$, $\widehat{\vb{y}}$ is a solution of \eqref{eq.aaug}.
    \end{claim}
    \begin{proof}
    By \Cref{prop.matrix-uau} and \Cref{prop.fin_inf_eq},
    \begin{equation}\label{eq.x_n_a_lim}
    \frac12 \left(\vb{x}^{(n)}_{\rho,A}\right)^\intercal A_\rho \left(\vb{x}^{(n)}_{\rho,A}\right)
    = \ecw{\rho}{\MG{A}{\vb{x}^{(n)}_{\rho,A}}} + O(n)
    = \frac{n^2}{2}g_\rho(A) + o(n^2). 
    \end{equation}
    Let $0 < n_1< n_2< \cdots$ be a sequence of positive integers such that
    $\lim_{k\to\infty} \frac1{n_k} \vb{x}^{(n_k)}_{\rho,A}=\widehat{\vb{y}}$. Then, \eqref{eq.x_n_a_lim}
    implies that $\widehat{\vb{y}}^\intercal A_\rho \widehat{\vb{y}} = g_\rho(A)$.
    Now for all integers $k$,
    write $\vb{x}^{(n_k)}_{\rho,A} = \left( x^{(n_k)}_1,\ldots,x^{(n_k)}_r \right)$,
    and let $v_1,v_2$ be two vertices
    in $\MG{A}{\vb{x}^{(n_k)}_{\rho,A}}$. By \Cref{prop.mat-zykov},
    all weighted degrees of the vertices are equal up to a difference of at most $\rho$.
    Using notation from \Cref{def.mat-graph}, assume without loss of generality that $v_1\in W_1$ and $v_2\in W_2$. Let $A=(U,D)$. Recall that $v_1, v_2$ have almost equal weighted degree. Taking the limit $k \to \infty$, we have quantitatively that
    \[\sum_{j=1}^r \left(\sym{A}{\rho}\right)_{1j} y_j 
    = \sum_{j=1}^r \left(\sym{A}{\rho}\right)_{2j} y_j,\]
    i.e., the first two coordinates of the vector
    $\left(\sym{A}{\rho}\right)\widehat{\vb{y}}$ are equal. Repeating this reasoning, we see that all coordinates in the vector are equal, i.e.,
    $\left(\sym{A}{\rho}\right)\widehat{\vb{y}}=c \vb 1$ 
    for some $c \in \RR_{\geq0}$.
    Since $g_\rho(A) = \widehat{\vb{y}}^\intercal A_\rho \widehat{\vb{y}} 
    = \widehat{\vb{y}}^\intercal (A_\rho)^\intercal \widehat{\vb{y}}$, we have
    \[g_\rho(A) = \widehat{\vb{y}}^\intercal (\sym{A}{\rho}) \widehat{\vb{y}} =
    \widehat{\vb{y}}^\intercal ( c,c,\dots,c) = c\cdot \norm{\widehat{\vb{y}}}_1 = c,\]
    which means indeed $(\sym{A}{\rho}) \widehat{\vb{y}} = g_\rho(A)\vb1$, so
    \eqref{eq.aaug} holds for $\widehat{\vb{y}}$.
\end{proof}

\begin{proof}[Proof of \Cref{lem.convergence}]
    Now we are ready to prove the lemma.
    Combining the three claims gives the desired conclusion:
    Since, $\frac 1n \vb{x}^{(n)}_{\rho,A} \in [0,1]^r$ for all $n$. the sequence $\left\{\frac1n \vb{x}_{\rho,A}^{(n)}\right\}_{n\in\ZZ_{>0}}$ has limit points;
    the limit $\lim_{n\to\infty} \frac1n \vb{x}_{\rho,A}^{(n)}$ exists; it is the unique solution to \eqref{eq.aaug}; 
    it has all positive coefficients; and it uniquely maximizes $\vb y^\intercal A_\rho \vb y$ subject to $\vb y \in \simplex{r}$.
\end{proof}

A simple corollary of \Cref{lem.convergence} is that the component sizes of $\MMG{A}{\rho}{n}$ (as in \Cref{def.mat-graph}) become arbitarily large as $n \to \infty$, since each term in $\optvec{A}{\rho}$ is positive. More precisely, we have the following:
\begin{corollary}\label{prop.mat-blowup}
    Let $\rho\in(1,\infty)$ and $A$ be a condensed mixed adjacency matrix with respect to $\rho$. For any positive integer $t$ there exists $N \in \ZZ_{\geq0}$ such that $\MG{A}{t\vb1} \subseteq \MMG{A}{\rho}{n}$ for all $n\geq N$.
\end{corollary}

We give $\vb y^*$ in \Cref{lem.convergence} its own name:
\begin{definition}[Optimal vector]\label{def.opt-vec}
  Let $\rho\in(1,\infty)$ and $A$ be a condensed mixed adjacency matrix with respect to $\rho$.
  We will denote the unique solution to \eqref{eq.aaug} by $\optvec{A}{\rho}$, the \emph{optimal vector} of $A$ with respect to $\rho$.
\end{definition}

We can now strengthen our earlier Zykov symmetrization  argument in a similar fashion to Lemma 1 of~\cite{brown_erdos_simonovits}.

\begin{proposition}[Lagrangians]\label{prop.lagrangians}
Fix $\rho\in(1,\infty)$ and let $A=(U,D)$ be a condensed mixed adjacency matrix with respect to $\rho$. For any $i\neq j$, if $U_{ii} = U_{jj}$ then $\left(\sym{A}{\rho}\right)_{ij} > U_{ii}$.
\end{proposition}
\begin{proof}
    Let $\optvec{A}{\rho} = (y_1,\dots, y_r)$. Without loss of generality we take $i = 1, j = 2$ with
    $$\sum_{k=3}^r \left(\sym{A}{\rho}\right)_{1k} y_k \geq \sum_{k=3}^r \left(\sym{A}{\rho}\right)_{2k} y_k.$$
    Let $\vb{y}'= (y_1+y_2,0,y_3,\dots,y_r)$. By \Cref{lem.convergence}, $\optvec{A}{\rho}$ is the unique optimal vector, so
    \begin{align*}
        0 & < \left(\optvec{A}{\rho}\right)^\intercal A_\rho \optvec{A}{\rho} - \left(\vb{y}'\right)^\intercal A \vb{y}'                                                                             \\
        & = y_2^2\left(U_{22}-U_{11}\right) + y_2 \sum_{k=3}^r 2\left(\left(\sym{A}{\rho}\right)_{2k} - \left(\sym{A}{\rho}\right)_{1k}\right)y_k + 2\big((\sym{A}{\rho})_{12}-U_{11}\big) y_1 y_2 \\
        & \leq2\big((\sym{A}{\rho})_{12} - U_{11}\big)y_1 y_2. \qedhere
    \end{align*}
\end{proof}
Importantly, the above implies that if all $U_{ii}$ are equal, then $\MG{A}{\vb 1}$ has \emph{complete underlying undirected graph}, meaning every pair of vertices is connected by some edge.

\subsection{Augmentation of mixed adjacency matrices}
Here, we strengthen \S3 of \cite{brown_erdos_simonovits}, generalizing the results to mixed graphs. Our primary goal is to find a small mixed adjacency matrix that yields a $\rho$-weighted blowup of high density. We will construct such mixed adjacency matrix by showing that we can grow a template by systematically adding vertices in a way that guarantees that the density increases at each step.
(In later sections we will analyze the procedure given here to ensure that it terminates after finitely many steps at an ``optimal'' mixed graph to blow up).

We term this addition process \emph{augmentation}; we begin with a condensed mixed adjacency matrix $A$ of size $r$, then add a row and column to $A$, equivalent to adding a new vertex part $W_{r+1}$ (as in \Cref{def.mat-graph}). This forms a new mixed adjacency matrix $A'$ of size $r+1$ and higher density.
The key result,~\Cref{lem.augmentation}, will enable us to find an asymptotically maximal mixed adjacency matrix by repeated augmentation.

\begin{definition}[Augmentation]\label{def.mat-augment}
  Fix $\rho\in(1,\infty)$.
  Let $B=(U,D)$ be a mixed adjacency matrix of size $r+1$ with $U_{(r+1)(r+1)}=0$.
  Let $A=(U',D')$ be the principal submatrix of $B$ obtained by removing the 
  $(r+1)$th row and column from both $U$ and $D$.
  Further suppose $A$ is condensed with respect to $\rho$, with optimal vector
  $\optvec{A}{\rho} = (y_1,\dots,y_r)$, and
  \begin{equation}\label{eq.aug_def}
    \sum_{j=1}^r \left(\sym{B}{\rho}\right)_{(r+1)j} y_j
    > g_\rho(A).
  \end{equation}
  Then we say that $B$ is obtained from $A$ by \emph{augmentation}, denoted by $A\augmentup B$.
\end{definition}

As alluded to earlier, we can see (via an identical argument to that in Lemma 2 of~\cite{brown_erdos_simonovits}) that the augmented matrix has higher density than the submatrix it arose from.
\begin{proposition}
    If $A\augmentup B$ then $g_\rho(A)<g_\rho(B)$.
\end{proposition}

In our process of repeated augmentation, we will repeatedly leverage the following result, which gives a useful condition for when we can augment a mixed adjacency matrix while preserving the property that a relatively large blowup of the matrix remains a subgraph of some specified mixed graph $G$ (we will later take $G$ to be a large $F$-free mixed graph, thereby using $G$ as a verification that our augmentations remain $F$-free). 

\begin{lemma}[Augmentation lemma]\label{lem.augmentation}
  Fix $\rho\in(1,\infty)$ and let $A$ be a condensed mixed adjacency matrix with respect to $\rho$,
  and $\eps$ be a positive real. For all $m \in \ZZ_{\geq0}$, 
  there exists $N = N(A,\rho,\eps,m) \in \ZZ_{\geq0}$ such that
  for any mixed graph $G$ on $n$ vertices ($n$ sufficiently large) that satisfies
  \begin{enumerate}[label=(\roman*)]
    \item $\degw v \geq (g_\rho(A)+\eps)n$ for all vertices $v\in V(G)$,
    \item $\MMG{A}{\rho}{N} \subseteq G$,
  \end{enumerate}
  there exists a mixed adjacency matrix $B$ such that $A\augmentup B$,
  and $\MMG{B}{\rho}{m} \subseteq G$.
\end{lemma}
\begin{proof}
  Take $N$ to be a fixed large integer (we precisely choose its value later in the argument).
  For any mixed graph $G$ on $n$ vertices that satisfies the lemma conditions,
  divide $V(G)$ into two parts: $V_1$, the vertices of some copy of
  $\MMG{A}{\rho}{N} \subset G$, and $V_2 = V(G) \backslash V_1$.
  Further partition $V_1 = W_1 \sqcup\dots\sqcup W_r$ (where $r$ is the size of $A$)
  according to~\Cref{def.mat-graph}.
  This partitioning is illustrated in \Cref{fig.aug-vertex-partition}.
  \begin{figure}[ht]
    \centering
    \includegraphics[width=10cm]{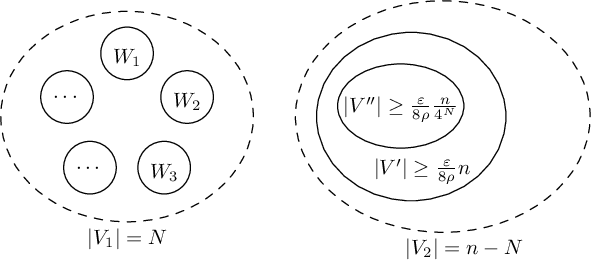}
    \caption{Partitioning $V(G)$ into $V_1$ and $V_2$}
    \label{fig.aug-vertex-partition}
  \end{figure}
  
  For vertex $v\in V(G)$ and set $S\subseteq V(G)$, we use
  $\degw(v,S)$ to denote the weighted degree of $v$ with respect to $S$.
  We also consider a weighted edge count between $V_1$ and $V_2$ defined similarly.
  \begin{claim}
    Let $V'\subseteq V_2$ be the set of all vertices $v\in V_2$ such that $\degw(v,V_1) \geq (g_\rho(A)+\frac\eps2)N$. Then $\abs{V'}\geq \frac{\eps}{8\rho}n$ for sufficiently large $n=n(\eps,\rho,N)$, and any mixed graph $G$ on $n$ vertices satisfying the two conditions of the lemma.
  \end{claim}
  \begin{proof}
    By condition (i), the weighted edge count between $V_1$ and $V_2$ is 
    \[\sum_{v\in V_2}\degw(v, V_1) =
    \sum_{u\in V_1}\degw(u, V_2) \geq
    (g_\rho(A)+\eps)(n-N) N -\rho N.\]
    When $n=n(\eps,\rho,N)$ is sufficiently large, $(g_\rho(A)+\eps)(n-N) N -\rho N \geq(g_\rho(A)+\frac{3\eps}4)(n-N) N$. Also note $\degw(v, V_1) \leq \rho N$ for all $v\in V_2$. Thus
    \begin{align*}
      \left(g_\rho(A)+\frac{3\eps}4\right)(n-N)N 
      \leq  \sum_{v\in V_2}\degw(v, V_1)
      & = \sum_{v\in V_2\setminus V'}\degw(v, V_1) + \sum_{v\in V'}(v, V_1) \\
      & \leq \left(g_\rho(A)+\frac{\eps}2\right)N(n-N) + \rho N \abs{V'},
    \end{align*}
    which simplifies to $\abs{V'} \geq \frac{\eps}{4\rho}(n-N) \geq \frac{\eps}{8\rho} n$.
  \end{proof}

  There are only $4^{N}$ distinct ways for a vertex in $V_2$ to be connected to
  $V_1$. Hence there exists
  $V''\subseteq V'$ with at least $\frac{\eps}{8\rho}\frac{n}{4^N}$ vertices, such that all vertices $v\in V''$ are
  connected identically to vertices in $V_1$, and
  $\degw(v, V_1) \geq \left(g_\rho(A)+\frac\eps 2\right)N$.

  \begin{figure}[ht]
    \centering
    \includegraphics[width=15cm]{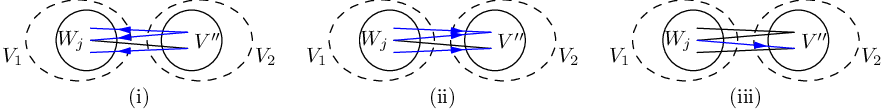}
    \caption{Illustrations of cases in the augmentation step}
    \label{fig.aug-case-ctor}
  \end{figure}

  Let $A=(U,D)$. We define matrix $B=(U',D')$ of size $r+1$ by adding a row and a column to $A$.
  Specifically,
  $U'_{(r+1)(r+1)} = D'_{(r+1)(r+1)}= 0$,
  and for each $j\in [r]$:
  \begin{enumerate}[label=(\roman*)]
    \item if there are at least $\frac{\eps}{8r\rho}N$ vertices in $W_j$ joined to each vertex in $V''$
          by a directed edge with head vertex in $W_j$, set $D'_{(r+1)j}=2$
          and $U'_{(r+1)j} = U'_{j(r+1)}=D'_{j(r+1)}=0$;
    \item else if there are at least $\frac{\eps}{8r\rho}N$ vertices in $W_j$ joined to each vertex in $V''$
          by a directed edge with head vertex in $V''$, set 
          $D'_{j(r+1)}=2$ and 
          $U'_{(r+1)j} = U'_{j(r+1)}=D'_{(r+1)j}=0$;
    \item else if there are at least $\frac{\eps}{8r\rho}N$ vertices in $W_j$ joined to each vertex
          of $V''$ (by either directed or undirected edges),
          set $U'_{(r+1)j} = U'_{j(r+1)} = 1$ and $D'_{(r+1)j} = D'_{j(r+1)} = 0$;
    \item else, set $U'_{(r+1)j} = U'_{j(r+1)} =D'_{(r+1)j} = D'_{j(r+1)} = 0$.
  \end{enumerate}
 See \Cref{fig.aug-case-ctor} for an illustration: in essence, we count the number of edges between each component and ``round'' this number to a linear factor of $N$ in order to decide the corresponding term of the mixed adjacency matrix.
  \begin{claim}\label{claim.aug-theta-ineq}
    For vertex $v\in V''$ and $j\in[r]$,
    the weighted degree of $v$ with respect to  $W_j$ is
    \begin{equation}\label{eq.deg-vcj}
    \degw(v, W_j) \leq \left(\sym{B}{\rho}\right)_{(r+1)j}\abs{W_j} 
    + \frac{\eps N}{4r}.
    \end{equation}
  \end{claim}
  \begin{proof}
    We examine all four cases in the construction of $B$ above. Since $\rho>1$, we have
    $\degw(v, W_j) \leq \rho\abs{W_j}$, so in cases (i) and (ii) where $\left(\sym{B}{\rho}\right)_{(r+1)j}= \rho$, \Cref{eq.deg-vcj} is clearly true.
    In case (iii), since $\left(\sym{B}{\rho}\right)_{(r+1)j}= 1$ and the number of directed edges is less than $2\frac{\eps}{8r\rho}N$, we have
    \[\degw(v, W_j) < \abs{W_j} + (\rho-1)\left(2\frac{\eps}{8r\rho}N\right) < \abs{W_j} + \frac{\eps N}{4r}.\]
    Finally in case (iv), the number of edges (directed or undirected) is less than $\frac{\eps}{8r\rho}N$; hence we have
    $\degw(v, W_j) < \rho \frac{\eps}{8r\rho}N< \frac{\eps N}{4r}$.
  \end{proof}

  \begin{claim}\label{claim.aug-a1}
    When $N=N(A,\rho,\eps)$ is sufficiently large, $A\augmentup B$.
  \end{claim}
  \begin{proof}
    By \Cref{claim.aug-theta-ineq}, for each vertex $v\in V''$, the weighted degree of $v$ with
    respect to $V_1$ is
    \[ \degw(v, V_1) \leq \sum_{j=1}^r 
      \left(\sym{B}{\rho}\right)_{(r+1)j}\abs{W_j}
      + \frac{\eps N}{4}.
    \]
    Combining this with $\degw(v, V_1)\geq(g_\rho(A)+\frac\eps2)N$, we get
    \[ \left(g_\rho(A)+\frac\eps2\right) N \leq \degw(v, V_1) \leq
      \sum_{j=1}^r\left(\sym{B}{\rho}\right)_{(r+1)j}\abs{W_j}
      + \frac{\eps N}{4}.
    \]
    Dividing by $N$,
    \[ g_\rho(A)+\frac\eps4 \leq \sum_{j=1}^r 
    \left(\sym{B}{\rho}\right)_{(r+1)j}\frac{\abs{W_j}}{N} .
    \]
    Let $\vb{y}^*_{\rho,A} = ( y_1,\dots,y_r)$ be the optimal vector of $A$.
    Recall that $W_j$ are the parts of $\MMG{A}{\rho}{N}$, so
    $\lim_{N\to\infty}\frac{\abs{W_j}}{N} = y_j$ by \Cref{lem.convergence}. So when $N=N(A,\rho,\eps)$ is sufficiently large,
    we have
    \[ g_\rho(A) < \sum_{j=1}^r 
     \left(\sym{B}{\rho}\right)_{(r+1)j} y_j,
    \]
    which means $A\augmentup B$ by \Cref{def.mat-augment}.
  \end{proof}
  
  Finally, we will show that $\MMG{B}{\rho}{m}\subseteq G$.
  Because $A$ is condensed, by \Cref{prop.mat-blowup}, if $N$ is sufficiently large, then 
  $\MG{A}{m\vb{1}}\subseteq\MMG{A}{\rho}{N}$.
  Thus $\MG{A}{m\vb{1}}\subseteq\MMG{A}{\rho}{N}\subseteq G$ by condition (ii) of the lemma. 
  By construction of $B$, when $n$ is
  sufficiently large, $\abs{V''}\geq m$ and all edges of $\MG{B}{m\vb{1}}$
  corresponding to the $(r+1)$th row and $(r+1)$th column of the undirected and directed parts
  of $B$ are in $G$.
  Therefore $\MG{B}{m\vb{1}}\subseteq G$. Since  $\MMG{B}{\rho}{m}\subseteq \MG{B}{m\vb{1}}$, we have   $\MMG{B}{\rho}{m}\subseteq G$.
\end{proof}

\subsection{Extremal mixed adjacency matrices}
We consider ``extremal'' mixed adjacency matrices whose blowups always yield $F$-free mixed graphs.

We begin with size $1$ mixed adjacency matrices and consider all matrices which can be obtained by successive augmentation series; we certify that there are only finitely many matrices, and that the one with highest density can be blown up into asymptotically extremal $F$-free mixed graphs.

We first define a notion of \emph{containment}: roughly speaking, the mixed adjacency matrix $A$ is contained in a sequence of mixed graphs $\{G^{(n)}\}$ if every mixed graph of the form $\MG{A}{\vb x}$ is a subgraph of some $G \in \{G^{(n)}\}$.
\begin{definition}[Containment]\label{def.contain}
  For $\rho\in(1,\infty)$, let $A$ be a condensed mixed adjacency matrix with respect to $\rho$,
  and let $\left\{G^{(n)}\right\}_{n\in\ZZ_{>0}}$ be a sequence of mixed graphs.
  We say $A$ is \emph{contained} in the sequence if for any positive
  integer $m$, there exists integer $n$ such that $\MMG{A}{\rho}{m} \subseteq G^{(n)}$.
\end{definition}

It is easy to see by \Cref{prop.mat-blowup} that the following three conditions are equivalent:
\begin{enumerate}[label=(\roman*)]
    \item $A$ is contained in the sequence $\left\{G^{(n)}\right\}_{n\in\ZZ_{>0}}$;
    \item there is an infinite sequence of integers $m_1 < m_2 < \cdots$ such that for each $m_i$ in the sequence, there exists integer $n$ so that $\MMG{A}{\rho}{m_i}\subseteq G^{(n)}$;
    \item for any positive integer $t$, there exists integer $n$ such that $\MG{A}{t\vb{1}} \subseteq G^{(n)}$.
\end{enumerate}

Now, following the proof of Theorem 1 in \cite{brown_erdos_simonovits}, we briefly discuss the Zarankiewicz problem \cite{Zarankiewicz1947SurLR} corresponding to the $\rho$-weighted extremal problem; the extremal mixed graphs that arise will in fact be the large mixed graphs $G$ on which we later apply the augmentation lemma (\Cref{lem.augmentation}).

\begin{definition}\label{def.mat-max-min-deg}
    Let $\rho\in(1,\infty)$, $F$ be a mixed graph, and $n$ be a positive integer. Let
    \[d^{(n)}_{\rho,F} := \max_{\substack{v(G)=n \\ F \not\subseteq G}} \left\{\min_{v\in V(G)} \degw v \right\}\quad,\quad
    Z^{(n)}_{\rho(F)} := \argmax_{\substack{v(G)=n \\ F \not\subseteq G}} \left\{\min_{v\in V(G)} \degw v \right\},\]
    and let $a_{\rho,F}^* := \limsup_{n\to\infty} \frac{d_{\rho,F}^{(n)}}n$.
    Then there exists an infinite sequence of increasing integers $N_{\rho,F} = (n_1, n_2, \dots)$ 
    such that 
    \[\lim_{j\to\infty} \frac{d_{\rho,F}^{(n_j)}}{n_j} = a_{\rho,F}^*.\]
\end{definition}
In other words $d^{(n)}_{\rho,F}$ is the maximal minimum-degree across all $n$ vertex $F$-free mixed graphs, and $Z^{(n)}_{\rho(F)}$ is a mixed graph achieving this maximum.
We will use $Z_{\rho,F}^{(n)}$ as the large $F$-free mixed graph $G$ in \Cref{lem.augmentation}; by construction all of its vertices have high minimum degree (equal to $d_{\rho,F}^{(n)}$).

Denote the two mixed adjacency matrices of size $1$ by $\vb 0 = ([0],[0])$ and $K = ([1],[0])$. The latter is named thus because $\MG{K}{t\vb1}$ is the complete graph $K_t$.

\begin{proposition}\label{prop.theta1r1}
  Let $F$ be a mixed graph with $\theta(F)\in(1,\infty)$. For any $\rho\in(1,\infty)$, $K$ is the only condensed mixed adjacency matrix with all diagonal elements of its undirected part equal to $1$ which can be contained in the sequence $\left\{Z_{\rho,F}^{(n)}\right\}_{n\in N_{\rho,F}}$.
\end{proposition}
\begin{proof}
    It suffices to show that any condensed mixed adjacency matrix $A=(U,D)$ of size at least $2$
    with all diagonal elements of $U$ equal to $1$ cannot satisfy this property.
    By \Cref{prop.lagrangians}, $\{D_{12},D_{21}\}=\{2,0\}$; without loss of generality suppose $D_{12}=2$.
    For any positive integer $t$, there exists $n\in N_{\rho,F}$ such that
    $\MG{A}{t\vb{1}}\subseteq Z_{\rho,F}^{(n)}$. We have $F\not\subseteq A[t\vb 1]$ since $Z_{\rho,F}^{(n)}$ is $F$-free by \Cref{def.mat-max-min-deg}. Now consider the mixed adjacency matrix 
    \[A' = \left(\begin{bmatrix} 1 & 0 \\ 0 & 0 \end{bmatrix}, 
    \begin{bmatrix} 0 & 2 \\ 0 & 0 \end{bmatrix}\right).\]
    We have $\MG{A'}{t\vb{1}}\subseteq\MG{A}{t\vb{1}}$, so
    $F\not\subseteq\MG{A'}{t\vb{1}}$ for all positive integers $t$,
    which means $F\not\subseteq M(x,s)$ (recall \Cref{def.graph-mxn}) for any positive integer $s$ and any $x\in(0,1)$.
    This implies $\theta(F) = 1$ by \Cref{prop.mxn-theta}, contradiction.
\end{proof}

Note that, since $\MMG{\vb 0}{\rho}{m}$ is simply an empty graph on $m$ vertices, $\vb 0$ is contained in $\left\{Z_{\rho,F}^{(n)}\right\}_{n\in N_{\rho,F}}$ for all mixed graphs $F$ and $\rho\in(1,\infty)$.
Based on the above proposition, we now define $\mathcal E_{\rho,F} = \{\vb 0, K\}$ if $K$ is contained in $\left\{Z_{\rho,F}^{(n)}\right\}_{n\in N_{\rho,F}}$, and $\{\vb 0\}$ otherwise.

We will leverage $\mcE_{\rho, F}$ to obtain an extremal $F$-free family of mixed graphs. We first suppose that $K \in \mcE_{\rho, F}$.
\begin{proposition}\label{prop.mat-max-min-deg-conv-d1}
  Let $\rho\in(1,\infty)$ and let $F$ be a mixed graph with $\theta(F)\in(1,\infty)$.
  If $K\in\mathcal E_{\rho,F}$ then the sequence $\frac{d_{\rho,F}^{(n)}}n$ converges and
  \[\lim_{n\to\infty} \frac{d_{\rho,F}^{(n)}}n = 1= g_\rho(K).\]
\end{proposition}
\begin{proof}
    Since $K$ is contained in the sequence $\left\{Z_{\rho,F}^{(n)}\right\}_{n\in N_{\rho,F}}$, 
    for any positive integer $t$ there exists $n\in N_{\rho,F}$ 
    such that the undirected complete graph $K_t = \MG{K}{t \vb1}\subseteq Z_{\rho,F}^{(n)}$, thus $F\not\subseteq K_t$ for all $t$, implying $d_{\rho,F}^{(t)}\geq t-1$ for all positive integers $t$. Hence
    \[\liminf_{n\to\infty} \frac{d_{\rho,F}^{(n)}}n \geq \liminf_{n\to\infty} \frac{n-1}n = 1.\]
    
    Suppose to the contrary that $\limsup_{n\to\infty} \frac{d_{\rho,F}^{(n)}}n > 1$, 
    then there exists
    $\eps > 0$ such that $d_{\rho,F}^{(n)} > (1+\eps) n$ for all sufficiently large $n \in N_{\rho,F}$.
    All conditions for \Cref{lem.augmentation} on $K$ and $Z_{\rho,F}^{(n)}$ are met for large $n$, so there exists mixed adjacency matrix $B$ such that $K\augmentup B$,
    and $\MMG{B}{\rho}{m}\subseteq Z_{\rho,F}^{(n)}$.
    This implies $F\not\subseteq\MMG{B}{\rho}{m}$. Since augmentation increases density, there are only two possible $B$ values:
    \[
    B=
    \left(\begin{bmatrix}
    1 & 0 \\ 0 & 0
    \end{bmatrix}, 
    \begin{bmatrix}
    0 & 2 \\ 0 & 0
    \end{bmatrix}\right)
    \text{ or }
    \left(\begin{bmatrix}
    1 & 0 \\ 0 & 0
    \end{bmatrix},
    \begin{bmatrix}
    0 & 0 \\ 2 & 0
    \end{bmatrix}
    \right).\]
    By the same reasoning as in the proof of \Cref{prop.theta1r1},
    we must have $F\not\subseteq M(x,t)$ for any positive integer $t$ and real number $x\in(0,1)$, thus $\theta(F)=1$, contradiction.
    Therefore $\limsup_{n\to\infty} \frac{d_{\rho,F}^{(n)}}n \leq 1$.
    
    Hence we conclude that $\lim_{n\to\infty} \frac{d_{\rho,F}^{(n)}}n = 1= g_\rho(K).$
\end{proof}

The case where $K\not\in\mathcal E_{\rho,F}$ is more complicated but follows the same premise.
\begin{definition}\label{def.b-family-ext}
  Let $F$ be a mixed graph with $\theta(F)\in(1,\infty)$. Given $\rho\in(1,\infty)$, 
  suppose $K \notin \mathcal E_{\rho,F}$.
  Define $\mathcal B_{\rho,F}$ as the set of all mixed adjacency matrices $B$ which are condensed with respect to $\rho$, and such that $B$ is contained
  in the sequence of mixed graphs $\left\{Z^{(n)}_{\rho,F}\right\}_{n \in N_{\rho,F}}$, and $B$ is obtainable from some finite (possibly empty) sequence
  \[\vb0 \augmentup B_1 \densesub B_2 \augmentup B_3 \densesub \cdots \densesub B_{2k}=B.\]
\end{definition}

We would like to choose the mixed adjacency matrix $B$ of maximal density from $\mathcal B_{\rho,F}$. The following proposition ensures that this is in fact possible:
\begin{proposition}[Finiteness]\label{prop.bf-finite}
  If $K\notin\mathcal E_{\rho,F}$, then the set $\mathcal B_{\rho,F}$ is finite.
\end{proposition}
\begin{proof}
    Suppose for the sake of contradiction that $\mathcal B_{\rho,F}$ is infinite. Note that the undirected part of each $B\in\mathcal B_{\rho,F}$ has all diagonal elements equal to zero, since augmentation only adds $0$s to the diagonals.
    
    Because there are only finitely many mixed adjacency matrices
    of any given size, there are arbitrarily large mixed adjacency matrices 
    $B\in\mathcal B_{\rho,F}$ contained in
    the sequence $\left\{Z_{\rho,F}^{(n)}\right\}_{n\in N_{\rho,F}}$. 
    
    Take $B \in \mathcal B_{\rho,F}$ of arbitrarily large size. Let $B'=(U',D')$ be the mixed adjacency matrix of the same size, with $U'$ having all $0$s on the diagonal and all $1$s elsewhere, and $D'$ having all elements $0$.
    By \Cref{prop.lagrangians}, each pair of vertices in $\MG{B}{\vb{1}}$ are connected by
    an edge, so $\MG{B'}{\vb{1}}\subseteq \MG{B}{\vb{1}}$. Hence the fact that $B$ is contained in 
    $\left\{Z_{\rho,F}^{(n)}\right\}_{n \in N_{\rho,F}}$ implies that $B'$ is as well,
    which means we can find arbitrarily large complete undirected graphs contained in
    $\left\{Z_{\rho,F}^{(n)}\right\}_{n \in N_{\rho,F}}$.
    Therefore $K$ is contained in $\left\{Z_{\rho,F}^{(n)}\right\}_{n \in N_{\rho,F}}$, implying $K\in \mathcal E_{\rho,F}$, contradiction.
\end{proof}

Now we give the aforementioned result on $\lim_{n\to\infty} \frac1n d_{\rho,F}^{(n)}$.
\begin{proposition}\label{prop.mat-max-min-deg-conv-theta}
  Let $F$ be a mixed graph with $\theta(F)\in(1,\infty)$, and $\rho\in(1,\infty)$.
  If $K\notin \mathcal E_{\rho,F}$, then the sequence $\frac{d_{\rho,F}^{(n)}}n$ converges and   
  \[\lim_{n\to\infty} \frac{d_{\rho,F}^{(n)}}n = \max_{B\in\mathcal B_{\rho,F}} \big\{g_\rho(B)\big\}.\]
\end{proposition}
\begin{proof}
    Let $B^*$ be any mixed adjacency matrix in $\mathcal B_{\rho,F}$ with maximal density.
    By \Cref{prop.fin_inf_eq} and \Cref{prop.mat-zykov}, given any positive integer $n$,
    all vertices $v$ in $\MMG{B^*}{\rho}{n}$ have weighted degree $\degw v = g_\rho\left(B^*\right)n + o(n)$.
    Since $F\not\subseteq \MMG{B^*}{\rho}{n}$, we have $d_{\rho,F}^{(n)}\geq g_\rho\left(B^*\right)n+o(n)$ for all $n$, and
    \[\liminf_{n\to\infty}\frac{d_{\rho,F}^{(n)}}{n}\geq g_\rho\left(B^*\right).\]
    On the other hand, suppose \[\limsup_{n\to\infty} \frac{d_{\rho,F}^{(n)}}n = \lim_{\substack{n \in N_{\rho,F} \\ n\to\infty}} \frac{d_{\rho,F}^{(n)}}n  > g_\rho\left(B^*\right).\]
    There exists $\eps > 0$ such that any sufficiently large $n \in N_{\rho,F}$ satisfies $d_{\rho,F}^{(n)} > \left(g_\rho\left(B^*\right)+\eps\right) n$.
    For any $m$, take $N = N (A, \rho, \eps, m)$ from \Cref{lem.augmentation}.
    There exists some $n \in N_{\rho,F}$ such that $\MMG{B^*}{\rho}{N}\subseteq Z_{\rho,F}^{(n)}$. We established earlier that $\degw v>\left(g_\rho\left(B^*\right)+\eps\right) n$ for all vertices $v$ of $Z_{\rho,F}^{(n)}$. Hence \Cref{lem.augmentation} on $B^*$ and $Z_{\rho,F}^{(n)}$ for large $n$, states that
    there exist mixed adjacency matrix $B'$ such that $B^*\augmentup B'$ (dependent on $m$ and $n=n(m)$) and $\MMG{B'}{\rho}{m}\subseteq Z_{\rho,F}^{(n)}$.
    We then take $B'\densesub B''$ then $B''$ is condensed and $\MMG{B''}{\rho}{m}\subseteq Z_{\rho,F}^{(n)}$.
    
    Note there are only finitely many such  $B''$ because the size of each is at most one more than the size of $B^*$. Thus there is at least one such  $B''$ such that $\MMG{B''}{\rho}{m}\subseteq Z_{\rho,F}^{(n)}$ is true for infinitely many pairs $m, n(m)$. This $B''$ is contained in the sequence $\left\{Z_{\rho,F}^{(n)}\right\}_{n\in N_{\rho,F}}$, so $B''\in \mathcal B_{\rho,F}$ and $g_\rho(B'')>g_\rho\left(B^*\right)$, contradicting the maximality of $g_\rho\left(B^*\right)$ in $\mathcal B_{\rho,F}$. Therefore
    \[\limsup_{n\to\infty}\frac{d_{\rho,F}^{(n)}}{n}\leq g_\rho(B^*).\]
\end{proof}

Combining the two cases, \Cref{prop.mat-max-min-deg-conv-d1} and \Cref{prop.mat-max-min-deg-conv-theta} allow for the following definition and corollary:
\begin{definition}[Extremal mixed adjacency matrix]\label{def.ext-matrix}
  Let $F$ be a mixed graph such that $\theta(F)\in(1,\infty)$, and let $\rho\in(1,\infty)$. We use $B_{\rho,F}^*$ to denote the \emph{extremal mixed adjacency matrix} for $F$ with respect to $\rho$, defined by the following:
  \begin{enumerate}[label=(\roman*)]
    \item $B_{\rho,F}^* \coloneqq K$ if $K\in \mathcal E_{\rho,F}$;
    \item $B_{\rho,F}^* \coloneqq \argmax_{B\in \mathcal B_{\rho,F}}\;g_\rho(B)$ otherwise.
  \end{enumerate}
\end{definition}
\begin{corollary}\label{prop.df-converge-ext}
  For a mixed graph $F$ with $\theta(F)\in(1,\infty)$, and for $\rho\in(1,\infty)$,
  the sequence $\frac{d_{\rho,F}^{(n)}}n$ converges to the density of its extremal mixed adjacency matrix, namely
  \[\lim_{n\to\infty} \frac{d_{\rho,F}^{(n)}}n =  g_\rho\left(B_{\rho,F}^*\right).\]
  Furthermore, $B_{\rho,F}^*$ is contained in the sequence of mixed graphs $\left\{Z_{\rho,F}^{(n)}\right\}_{n\in N_{\rho,F}}$ and therefore 
  $F\not\subseteq \MMG{B_{\rho,F}^*}{\rho}{n}$ for any positive integer $n$.
\end{corollary}

Our aim is now to show that the maximal graphs obtained by optimally blowing up $B_{\rho,F}^*$ are asymptotically extremal. We first make this notion precise in \Cref{def.extremal}, then prove it in \Cref{lem.extremal}.
\begin{definition}[Asymptotically extremal sequence]\label{def.extremal}
  Let $F$ be a mixed graph; we say that the sequence of mixed graphs $\left\{G^{(n)}\right\}_{n\in\ZZ_{>0}}$, where each $G^{(n)}$ has $n$ vertices, is \emph{asymptotically extremal} for $F$ (with respect to $\rho$) if
  \begin{itemize}
      \item $F\not\subseteq G^{(n)}$ for any $n$;
      \item for any $\eps>0$, for sufficiently large $n$, any $n$-vertex mixed graph $G$ such that $\ecw{\rho}{G} > \ecw{\rho}{G^{(n)}} + \eps n^2$ must satisfy $F\subseteq G$.
  \end{itemize}
\end{definition}

\begin{lemma}[Extremal lemma]\label{lem.extremal}
  Let $F$ be a mixed graph with $\theta(F)\in(1,\infty)$, and  $\rho\in(1,\infty)$. The sequence of maximal mixed graphs $\left\{\MMG{B_{\rho,F}^*}{\rho}{n}\right\}_{n\in\ZZ_{>0}}$ is asymptotically extremal for $F$.
\end{lemma}
\begin{proof}
    We use $B^*$ to denote $B_{\rho,F}^*$ in this proof.
    Let $\eps$ be a positive real and $G$ be an $n$-vertex mixed graph (for large $n$) such that $\ecw{\rho}{G} > \ecw{\rho}{\MMG{B^*}{\rho}{n}} + \eps n^2$.
    By \Cref{prop.fin_inf_eq}, when $n$ is sufficiently large,
    $\ecw{\rho}{\MMG{B^*}{\rho}{n}}>\left(g_\rho(B^*)-\frac\eps2\right)\frac{n^2}{2}$,
    so $\ecw{\rho}{G}> \left(g_\rho(B^*)+\frac{3\eps}2\right)\frac{n^2}{2}$.
    
    \begin{claim}\label{claim.g-n-finite-seq}
    When $n$ is sufficiently large, there exists a subgraph $H\subseteq G$ on $k$ vertices where
    $\sqrt{\frac{\eps}{2\rho}}n \leq k \leq n$, and
    \[\min_{v\in V(H)} \degw v > \left(g_\rho(B^*)+\frac\eps2\right)k.\]
    (Here $\degw v$ refers to the weighted degree of $v$ in $H$.)
    \end{claim}
    \begin{proof}
    Construct a sequence of mixed graphs $G^{(n)}, G^{(n-1)}, \dots$ such that $G^{(n)}=G$
    and $G^{(n-j)}$ is obtained from $G^{(n-j+1)}$ by removing a vertex $v$ with
    $\degw v \leq \left(g_\rho(B^*)+\frac\eps2\right)(n-j+1)$. The sequence terminates when such a vertex does not exist.
    Then when $n$ is sufficiently large, $G^{(n-j)}$ in the sequence has
    \begin{align*}
        \ecw{\rho}{G^{(n-j)}}
        & > \left(g_\rho(B^*)+\frac{3\eps}2\right)\frac{n^2}2 - \left(g_\rho(B^*)+\frac\eps2\right)\frac{j(2n-j+1)}{2} \\
        & \geq \left(g_\rho(B^*)+\frac\eps2\right) \frac{(n-j)^2}2 +
        \frac12\eps n^2 + O(n)                                                                                          \\
        & \geq \frac14\eps n^2.
    \end{align*}
    But $\ecw{\rho}{G^{(n-j)}} \leq \rho \binom{n-j}{2}$,
    so $\rho\frac{(n-j)^2}{2}> \frac14\eps n^2$, which implies
    $n-j \geq \sqrt{\frac{\eps}{2\rho}}n$, i.e., the sequence of mixed graphs must terminate
    at some graph $H$ on $k\geq \sqrt{\frac{\eps}{2\rho}}n$ vertices.
    Then, by construction,
    \[\min_{v\in V(H)} \degw v > \left(g_\rho(B^*)+\frac\eps 2\right)k.\]
    \end{proof}
    
    Now we can finish the proof of the lemma. For the sake of contradiction, assume there exists
    $\eps>0$, an infinite sequence of integers $n_1 < n_2 < \cdots$, and mixed graphs $G_j$
    on $n_j$ vertices, such that for all sufficiently large $j$,
    $\ecw{\rho}{G_j} > \ecw{\rho}{\MMG{B^*}{\rho}{n_j}}+\eps n_j^2$ and $F\not\subseteq G_j$.
    By \Cref{claim.g-n-finite-seq},  there is a subgraph $H_j\subseteq G_j$
    (therefore $F\not\subseteq H_j$) on $k_j$ vertices with
    $\sqrt{\frac{\eps}{2\rho}} n_j \leq k_j \leq n_j$ and
    $\min_{v\in V(H_j)} \degw v > \left(g_\rho(B^*)+\frac\eps2\right)k_j$.
    Therefore $\frac{d^{(k_j)}_{\rho,F}}{k_j} > g_\rho(B^*)+\frac\eps2$ for all sufficiently large $j$
    (recall \Cref{def.mat-max-min-deg}). Thus
    \[\limsup_{n\to\infty} \frac{d_{\rho,F}^{(n)}}{n} \geq \limsup_{j\to\infty}\frac{d_{\rho,F}^{(k_j)}}{k_j} \geq g_\rho(B^*)+\frac\eps2.\]
    But this contradicts \Cref{prop.df-converge-ext} which states that $\lim_{n\to\infty} \frac{d_{\rho,F}^{(n)}}n = g_\rho(B^*)$.
\end{proof}

Finally, note that while the set $\mathcal B_{\rho,F}$ is finite, its size can vary based on the value of $\rho$. We now must relate this to $\mathcal M_F$, as defined in \Cref{subsec.weightedextremal}.

\begin{proposition}\label{prop.ext-in-mf}
    Let $F$ be a mixed graph with at least one directed edge and $\theta(F)\in(1,\infty)$. Then $B_{\rho,F}^* \in \mathcal M_F$ for any $\rho\in(1,\infty)$.
\end{proposition}
\begin{proof}
    If $K \in \mathcal E_{\rho,F}$, then $B_{\rho,F}^* = K \in \mathcal M_F$ by \Cref{def.ext-matrix}.
    Otherwise $K \not\in \mathcal E_{\rho,F}$. Let $B_{\rho,F}^* = (U,D)$.
    
    Conditions (i) and (iii) of \Cref{def.mf-no-rho} are met by \Cref{def.b-family-ext}, and condition (v) is met by \Cref{prop.lagrangians}.
    Note that if $D$ is the zero matrix then $g(B_{\rho,F}^*) = \max_{\vb y\in\simplex{r}} \vb{y}^\intercal U \vb y < 1$, contradiction, so (iv) is true.
    By (iv) we know $\MG{B_{\rho,F}^*}{\vb 1}$ is a complete mixed graph with at least one directed edge. Hence its size is at most $\undirectedchi{F^\rhd}-1$, since otherwise for large enough $t'$ and $t$, we would have $F \subseteq F^\rhd[t] \subseteq \MG{B_{\rho,F}^*}{t'\vb1}$, contradiction. This means (ii) is true as well, so all conditions are satisfied and $B_{\rho,F}^* \in \mathcal M_F$.
\end{proof}

The main result follows:

\begin{proof}[Proof of \Cref{thm.matrix-maximal}]
    By \Cref{lem.extremal},
    \[
    \limsup_{\substack{v(G)\to\infty \\ F\not\subseteq G}} \ecw{\rho}{G}/\tbinom n2 = g_\rho(B_{\rho,F}^*).
    \]
    By \Cref{prop.ext-in-mf}, we have $B_{\rho,F}^* \in \mathcal M_F$. Furthermore, $B$ is $F$-free for all $B\in \mathcal M_F$; thus \Cref{lem.extremal} implies $g_\rho(B) \leq g_\rho(B_{\rho,F}^*)$. Hence
    \[
    \max_{B\in\mathcal M_F} g_\rho(B) = g_\rho(B_{\rho,F}^*),
    \]
    and the theorem is proved.
\end{proof}

\end{appendices}

\end{document}